\pdfoutput=1

\documentclass[a4paper,reqno]{article}

\usepackage[utf8]{inputenc} 
\usepackage[T1]{fontenc} 
\usepackage{lmodern} 
\usepackage[american]{babel}
\usepackage{csquotes}

\usepackage{verbatim} 
\usepackage{authblk} 
\usepackage{calc}
\usepackage{amsmath}
\usepackage{amsfonts}
\usepackage{graphicx}
\usepackage{float}
\usepackage{xcolor}
\usepackage{enumitem}
\setlist[enumerate]{leftmargin=.5in} 
\setlist[itemize]{leftmargin=.5in} 
\providecommand{\keywords}[1]{\textbf{Keywords: } #1} 

\def\natural{\mathbb{N}}
\def\integer{\mathbb{Z}}
\def\real{\mathbb{R}}

\newcommand{\qspace}[2]{\raisebox{0.7ex}{\ensuremath{#1}}\ensuremath{\mkern-3mu}\Big/\ensuremath{\mkern-3mu}\raisebox{-0.7ex}{\ensuremath{#2}}} 
\let\originalleft\left
\renewcommand{\left}{\mathopen{}\mathclose\bgroup\originalleft} 
\let\originalright\right\renewcommand{\right}{\aftergroup\egroup\originalright} 

\usepackage{amsthm}
\usepackage{amssymb}
\newtheorem*{theorem*}{Theorem} 
\newtheorem{theorem}{Theorem}[section] 
\newtheorem{lemma}[theorem]{Lemma} 
\newtheorem{proposition}[theorem]{Proposition}
\newtheorem{definition}[theorem]{Definition}
\newtheorem{example}[theorem]{Example}
\newtheorem{corollary}[theorem]{Corollary}

\usepackage{tikz}
\usetikzlibrary{arrows}
\usepackage{environ} 
\makeatletter
\newsavebox{\measure@tikzpicture}
\NewEnviron{scaletikzpicturetowidth}[1]{%
	\def\tikz@width{#1}%
	\begin{lrbox}{\measure@tikzpicture}%
		\BODY
	\end{lrbox}%
	\pgfmathparse{#1/\wd\measure@tikzpicture}%
	\BODY
}
\makeatother

\usepackage[
pdfmenubar=true,
pdftitle={Convergence Proof of Ginelli Algorithm},
pdfauthor={Florian Noethen},
colorlinks=true,
linkcolor=black,
citecolor=blue,
filecolor=cyan,
urlcolor=magenta,
raiselinks=false,
hyperfootnotes=true
]{hyperref}
\usepackage[capitalize,nameinlink]{cleveref} 

\begin{document}
	
	\title{A projector-based convergence proof of the Ginelli algorithm for covariant Lyapunov vectors}
\author{Florian Noethen
	\thanks{Fachbereich Mathematik, Universität Hamburg, Bundesstraße 55, 20146 Hamburg, Germany (\href{mailto:florian.noethen@uni-hamburg.de}{florian.noethen@uni-hamburg.de}).}
}
\date{February 27, 2019} 
\maketitle

\begin{abstract}
	Linear perturbations of solutions of dynamical systems exhibit different asymptotic growth rates, which are naturally characterized by so-called covariant Lyapunov vectors (CLVs). Due to an increased interest of CLVs in applications, several algorithms were developed to compute them. The Ginelli algorithm is among the most commonly used. Although several properties of the algorithm have been analyzed, there exists no mathematically rigorous convergence proof yet.\par
	In this article we extend existing approaches in order to construct a projector-based convergence proof of Ginelli's algorithm. One of the main ingredients will be an asymptotic characterization of CLVs via the Multiplicative Ergodic Theorem. In the proof, we keep a rather general setting allowing even for degenerate Lyapunov spectra.
\end{abstract}


\hspace{1em}



\tableofcontents

\hypersetup{linkcolor=red}
	
	\section{Introduction}\label{sectionIntroduction}
	In this paper, we provide a projector-based convergence proof of the Ginelli algorithm \cite{Ginelli2007}, which computes the so-called CLVs. They form an intrinsic basis of the tangent space along a given reference trajectory and, thus, describe its local model. In fact, CLVs can be seen as a generalization to eigenvectors describing the local model of a steady-state.\par
	By the Hartman-Grobman theorem the local model of a steady-state is linked with the original system. Eigenspaces correspond to invariant manifolds of the flow and eigenvalues indicate exponential growth/decay rates of perturbations of the equilibrium. Similar relations can be established for periodic orbits via Floquet theory.\par
	In 1968 Oseledets managed to find a suitable generalization that goes beyond the analysis of steady-states and periodic orbits. In his celebrated \textit{Multiplicative Ergodic Theorem} (MET) \cite{Oseledets1968} the long-term behavior of linear perturbations of arbitrary trajectories is explained. Similar to the case of steady-states, the tangent space is split into invariant subspaces that capture directions of different asymptotic growth rates. Instead of a splitting into eigenspaces, we obtain the \textit{Oseledets splitting} with its corresponding \textit{Lyapunov spectrum} consisting of \textit{Lyapunov exponents} (LEs) as opposed to eigenvalues. In the nondegenerate case, i.e. if the Lyapunov spectrum is simple, the Oseledets spaces are one-dimensional and, hence, can be identified with a basis of vectors for each point of the trajectory. Those vectors are called \textit{covariant Lyapunov vectors} (CLVs).\par
	Despite their prominent role, it was not until a few years ago that first algorithms to compute CLVs were developed. Following Ginelli's algorithm \cite{Ginelli2007} in 2007 several other approaches emerged \cite{Wolfe2007, Froyland2013, Kuptsov2012}, some of which are explained only for nondegenerate scenarios. With computational tools like Ginelli's algorithm at hand, CLVs became a frequent interest in applications. Amongst others, CLVs reveal structures in turbulent flows \cite{Conti2017,Inubushi2015} and are used to analyze hard-disk systems \cite{Bosetti2010,Bosetti2013,Morriss2013,Truant2014} and climate models \cite{Schubert2015, Schubert2016, Vannitsem2016}. Moreover, they constitute a hyperbolic decoupling of the tangent space of dissipative systems that extracts the physically relevant modes \cite{Takeuchi2011}. Furthermore, the angle between CLVs is used as an indicator for critical transitions in long-term behavior of solutions \cite{Beims2016,Sharafi2017} and as a degree of hyperbolicity \cite{Conti2017,Saiki2010,Xu2016,Yang2009} in dynamical systems. However, despite the existence of numerous applications, many theoretical aspects of CLV-algorithms are still unexplored. This paper is a step to reducing the gap between theory and applications. Our goal is to verify convergence of Ginelli's algorithm by correcting and extending previous results.\par
	In 1998 Ershov and Potapov investigated what could be called the first phase of Ginelli's algorithm, where past states of a reference point are explored to compute the fastest growing directions \cite{Ershov1998}. 15 years later Ginelli et al. built upon the work of Ershov and Potapov to formulate a convergence proof of their full algorithm \cite{Ginelli2013}. They focused on a second phase, where future states are probed to obtain the fastest decaying directions. By a certain relation between both phases it is possible to extract the CLVs.\par
	While \cite{Ershov1998} and \cite{Ginelli2013} present fundamental ideas on convergence of Ginelli's algorithm, we find it necessary to be more precise in some arguments. In particular, \cite{Ershov1998} shows that almost all initial vectors propagated from present to future will align with left singular vectors of the propagator asymptotically. Then, an estimate for propagation from past to present is obtained by shifting the estimate for propagation from present to future. This argument requires a more detailed analysis of the $O(1)$-terms as the condition on exceptional vectors that will not yield convergence depends on the starting point, which is neglected in \cite{Ershov1998}. In particular, the set of admissible initial vectors can be different for each starting point that is associated with a chosen runtime. Additionally, we find that both phases of Ginelli's algorithm should be treated as connected. Whereas, until now perfect convergence of the first phase was assumed to simplify the analysis of the second phase. However, despite the criticism, both papers are significant steps to better understand the connection between Oseledets MET and the Ginelli algorithm. In fact, they inspire many ideas presented here.\par
	Our new convergence proof fills missing details and even extends the existing results. Unlike in a nondegenerate scenario, we do not pose any restrictions on the Lyapunov spectrum. Arbitrary dimensions for Oseledets spaces are allowed. Moreover, we distinguish between a discrete and continuous time version of the algorithm. It turns out that both versions converge, however the precise notion of convergence is different. Namely, the discrete time version of Ginelli's algorithm converges for almost all configurations of initial vectors, whereas the continuous time version only converges in measure. Furthermore, by incorporating the Lyapunov index notation we find an estimate for the speed of  convergence. As already predicted and observed \cite{Ershov1998,Froyland2013,Ginelli2013,Vannitsem2016}, the speed of convergence is exponential with a rate determined by the minimal distance of LEs.\\
	
	The main part of our article is divided into three sections. \Cref{sectionTools} sets the notation and constructs tools needed for the convergence proof later on. A special interest lies in the evolution of vectors/subspaces in terms of distances and angles. In particular, the relation of propagated vectors to singular vectors is of importance, since singular vectors form directions of optimal growth rates for finite time. In \cref{sectionGinellisAlgorithm} we present Ginelli's algorithm and state a deterministic version of Oseledets MET. By having a fairly general setting, we try to include as many scenarios as possible, though, we assume finite dimensional dynamics. With all preparations finished, we are in a position to precisely formulate and prove convergence of Ginelli's algorithm. The main work of \cref{sectionConvergence} consists in assembling the tools obtained in \cref{sectionTools}, while the MET from \cref{sectionGinellisAlgorithm} serves as an interface between evolution of singular vectors and CLVs.
	\bigskip
	
		\section{Notation and Tools}\label{sectionTools}
	This section is primarily concerned with the evolution of vectors and subspaces. In order to keep track of the speed of convergence, we define the notion of a Lyapunov index. Next, we set up necessary notation to describe distances and angles of subspaces. In particular, we are interested in how those quantities change after applying a propagation map and after orthogonalization, e.g. the Gram-Schmidt procedure. An estimate of the rate of change is given based on a relation to singular vectors of the propagating linear map. As it turns out, there are configurations of vectors and subspaces that perform better than others. A distinction between them will be made by introducing a so-called admissibility parameter. Later on we will use the admissibility parameter to describe how well a configuration behaves in Ginelli's algorithm.

	\subsection{Lyapunov Index}\label{subsectionLyapunovIndex}
	When analyzing an algorithm, one of the main aspects to consider is the speed of convergence. It is defined as the rate of change of the distance between a current and a sough-after state as a parameter, such as time, is increased. In our case time can be either discrete ($\mathbb{T}=\integer$) or continuous ($\mathbb{T}=\real$). Moreover, the nature of the problem or features of the algorithm might already prescribe certain timescales. In fact, LEs and CLVs describe properties on an exponential time scale, which can be captured by the Lyapunov index notation.
	\begin{definition}\label[definition]{definitionLyapunovIndex}
		The \emph{Lyapunov index} $\lambda(f)\in \real\cup\{\pm\infty\}$ of a function $f:\mathbb{T}_{\geq 0}\to\real_{\geq 0}$ is defined as the limit
		\begin{equation*}
		\lambda(f):=\limsup_{t\to\infty}\frac{1}{t}\log f(t).
		\end{equation*}
	\end{definition}
	Roughly speaking, the function $f$ behaves similar to $e^{t\lambda(f)}$ on an exponential scale. For example, a negative Lyapunov index implies exponential decay. However, one should note that variations on smaller scales are not included in this notation\footnote{For example, $e^{-t}$ and $\sin(t)t^2e^{-t}$ have the same Lyapunov index.}, but very well may be of importance for limited time scenarios such as numerical computations.\par 
	Next, we list some useful properties for the Lyapunov index, which can be found in Arnold's book \cite{Arnold1998} and are easily verified:
	\begin{proposition}\label[proposition]{propositionLyapunovIndexProperties}
		Let $f,g:\mathbb{T}_{\geq 0}\to\real_{\geq 0}$. The following are true:
		\begin{enumerate}
			\item $\lambda(0)=-\infty$,
			\item $\lambda(c)=0$ for $c>0$ constant,
			\item $\lambda(\alpha f)=\lambda(f)$ for $\alpha>0$,
			\item $\lambda\left(f^\alpha\right)=\alpha\lambda(f)$ for $\alpha>0$,
			\item $f\leq g\implies\lambda(f)\leq\lambda(g)$,
			\item $\lambda(f+g)\leq \max(\lambda(f),\lambda(g))$,
			\item $\lambda(fg)\leq\lambda(f)+\lambda(g)$ (if the right-hand side makes sense).
		\end{enumerate}
	\end{proposition}
	 As the algorithm consists of two subsequent phases, the Lyapunov index is not enough for discussing Ginelli's algorithm. Each phase has its own runtime that influences the resulting approximation. For a good approximation, both runtimes need to be increased. Certainly, there are circumstances and rules that prescribe a favoring relation between those runtimes. However, we will not discuss them here. Instead, we settle for a formulation that allows two different runtimes. For this purpose, we extend the notion of a Lyapunov index to a formulation depending on two parameters.
	\begin{definition}\label[definition]{definitionExtendedLyapunovIndex}
		The \emph{extended Lyapunov index} $\overline{\lambda}(f)\in \real\cup\{\pm\infty\}$ of a function $f:\mathbb{T}_{\geq 0}\times\mathbb{T}_{\geq 0}\to\real_{\geq 0}$ is defined as the limit
		\begin{equation*}
		\overline{\lambda}(f):=\limsup_{T\to\infty}\,\sup_{t_1,t_2\geq T}\frac{1}{\min(t_1,t_2)}\log f(t_1,t_2).
		\end{equation*}
	\end{definition}
	In contrast to the standard Lyapunov index, the new quantity describes behavior on an exponential timescale as $\min(t_1,t_2)$ is increased. Especially, when fixing a certain relation between both parameters, an upper bound on the speed of convergence is given by the extended Lyapunov index.\footnote{For example, given the relation $t_1=2t_2$ we have $\lambda(f(2t,t))\leq\overline{\lambda}(f)$.} In fact, the extended version exhibits similar properties to the usual Lyapunov index.
	\begin{proposition}\label[proposition]{propositionExtendedLyapunovIndexProperties}
		Rules 1-7 of \textnormal{\cref{propositionLyapunovIndexProperties}} hold true with $\lambda$ replaced by $\overline{\lambda}$. Furthermore, if we extend a function $f:\mathbb{T}_{\geq 0}\to\real_{\geq 0}$ to $\overline{f}:\mathbb{T}_{\geq 0}\times\mathbb{T}_{\geq 0}\to\real_{\geq 0}$ by setting $\overline{f}(t_1,t_2):=f(t_1)$, then
		\begin{enumerate}
			\item[8.] $\lambda(f)<0\implies\overline{\lambda}\left(\overline{f}\right)=\lambda(f)$.
		\end{enumerate}
	\end{proposition}
	\begin{proof}
		Rules 1,2,4,5 and 7 follow directly from the definition. To show rule 3, we have $f\leq\alpha f$ for $\alpha\geq 1$, and hence
		\begin{equation*}
		\overline{\lambda}(f)\leq\overline{\lambda}(\alpha f)\leq \overline{\lambda}(\alpha)+\overline{\lambda}(f)=\overline{\lambda}(f).
		\end{equation*}
		The case $0<\alpha<1$ follows by looking at $\beta:=\frac{1}{\alpha}$ and $g:=\alpha f$. Moreover, it is easily verified that
		\begin{equation*}
		\overline{\lambda}(f+g)\leq\overline{\lambda}(2\max(f,g))=\overline{\lambda}(\max(f,g))=\max\left(\overline{\lambda}(f),\overline{\lambda}(g)\right).
		\end{equation*}
		Now, let $\overline{f}$ be the extension of some function $f:\mathbb{T}_{\geq 0}\to\real_{\geq 0}$ as above. The relation $\lambda(f)\leq\overline{\lambda}(\overline{f})$ is always satisfied. To show equality, we remark that $\lambda(f)<0$ implies the existence of some $T>0$ with $\log f(t)<0$ for all $t\geq T$. In particular, it holds
		\begin{equation*}
		\sup_{t_1,t_2\geq t}\frac{1}{\min\left(t_1,t_2\right)}\log f\left(t_1\right)\leq \sup_{t_1\geq t}\frac{1}{t_1}\log f\left(t_1\right)
		\end{equation*}
		with right-hand side converging to $\lambda(f)$ for $t\to\infty$.		
	\end{proof}
	We demonstrate two exceptional cases where the function is growing/decaying either too slow or too fast to be captured by the notation.
	\begin{example}\label[example]{exampleExtendedLyapunovIndex}
		Let $f(t_1,t_2):=\left\lceil\min(t_1,t_2)^2\right\rceil$ and $g(t_1,t_2):=\alpha^{2f(t_1,t_2)-1}$ for $0<\alpha<1$. We compute
		\begin{equation*}
		0=\overline{\lambda}\left(1\right)
		\leq \overline{\lambda}\left(f\right)
		\leq \overline{\lambda}\left(\min(t_1,t_2)^2+1\right)
		\leq\max\left(2\overline{\lambda}\left(\min(t_1,t_2)\right),0\right)
		=0
		\end{equation*}
		and
		\begin{equation*}
		\overline{\lambda}(g)=\overline{\lambda}\left(\frac{1}{\alpha}\left(\alpha^{f}\right)^2\right)
		=2\overline{\lambda}\left(\alpha^{f}\right)
		\leq 2\overline{\lambda}\left(\alpha^{\min(t_1,t_2)^2}\right)
		=-\infty.
		\end{equation*}
	\end{example}

	\subsection{Orthogonal Projections}\label{subsectionOrthogonalProjections}
	We present some essential results about orthogonal projection. For most facts, we specifically refer to chapter 1.6 of Kato's book \cite{Kato2013}, the chapter on projections in Galántai's book \cite{Galantai2004} and the chapter by Deutsch \cite{Deutsch1995}.\par
	Amongst others, orthogonal projections are a tool to describe geometric properties of subspaces. We associate a subspace $M\subset\real^d$ and its corresponding orthogonal projection $P_M$ using the standard inner product. Through this identification we can define distances and angles between subspaces, or even speak of converging sequences of subspaces. Since we focus on the euclidean norm $\|\cdot\|_2$, let us drop the subscript and simply write $\|\cdot\|$.\footnote{As all norms on $\real^d$ are equivalent, quantities that are defined on an exponential scale remain the same. In particular, LEs and CLVs are independent of the chosen norm. Moreover, estimates on the exponential speed of convergence in our main theorems in \cref{sectionConvergence} remain unchanged.}
	\begin{definition}\label[definition]{definitionDistance}
		The \emph{distance} between two subspaces $M,N\subset\real^d$ is defined as 
		\begin{equation*}
		d(M,N):=\|P_M-P_N\|.
		\end{equation*}
	\end{definition}
	We state a collection of handy properties mostly from \cite{Galantai2004}.
	\begin{proposition}\label[proposition]{propositionDistance}
		The distance $d$ is a metric on the set of subspaces. Moreover, the following holds for all subspaces $M,N\subset\real^d$:
		\begin{enumerate}
			\item $0\leq d(M,N) \leq 1$,
			\item $d(M,N)=d\left(M^{\perp},N^{\perp}\right)$,
			\item $d(M,N)<1\ \Rightarrow\ \dim(M)=\dim(N)$.
		\end{enumerate}
		In case that $\dim(M)=\dim(N)$, we also have:
		\begin{enumerate}
			\item[4.] $d(M,N)=\|P_M P_{N^{\perp}}\|$,
			\item[5.] $d(M,N)=1\ \iff\ M\cap N^{\perp}\neq\{0\}$.
		\end{enumerate}
		If $V\in O(d,\real)$ is an orthogonal transformation, then
		\begin{enumerate}
			\item[6.] $d(V(M),V(N))=d(M,N)$.
		\end{enumerate}
	\end{proposition}
	Every invertible linear map induces a Lipschitz-continuous transformation of the set of subspaces.
	\begin{corollary}\label[corollary]{corollaryInducedTransformationLipschitz}
		For each $A\in \textnormal{Gl}(d,\real)$ and all subspaces $M,N\subset\real^d$, we have
		\begin{equation*}
		d(A(M),A(N)) \leq \|A\|\,\|A^{-1}\|\,d(M,N).
		\end{equation*}
	\end{corollary}
	\begin{proof}
		Fix an invertible map $A$. For subspaces of different dimension, the inequality is trivially satisfied. So, let $M$ and $N$ be of the same dimension. We compute:
		\begin{align*}
		d(A(M),A(N))&=\|P_{A(M)}P_{(A(N))^{\perp}}\|\\
		&=\|P_{A(M)}P_{(A^*)^{-1}N^{\perp}}\|\\
		&=\max_{x\in M\setminus\{0\},\, y\in N^{\perp}\setminus\{0\}}\frac{|\langle Ax,(A^*)^{-1}y\rangle|}{\|Ax\|\,\|(A^*)^{-1}y\|}\\
		&=\max_{x\in M\setminus\{0\},\, y\in N^{\perp}\setminus\{0\}}\frac{|\langle x,y\rangle|}{\|x\|\,\|y\|}\,\frac{\|A^{-1}(Ax)\|}{\|Ax\|}\,\frac{\|A^*((A^*)^{-1}y)\|}{\|(A^*)^{-1}y\|}\\
		&\leq \max_{x\in M\setminus\{0\},\, y\in N^{\perp}\setminus\{0\}}\frac{|\langle x,y\rangle|}{\|x\|\,\|y\|}\,\|A^{-1}\|\,\|A^*\|\\
		&=\|A\|\,\|A^{-1}\|\,\|P_{M}P_{N^{\perp}}\|\\
		&=\|A\|\,\|A^{-1}\|\, d(M,N).
		\end{align*}
		Here, $A^*$ denotes the adjoint map of $A$ with respect to the standard inner product.
	\end{proof}
	The next concept needed is the (minimal) angle between two subspaces. A lot on this topic can be found in \cite{Deutsch1995}.
	\begin{definition}\label[definition]{definitionAngle}
		The cosine of the \emph{angle} between $M$ and $N$ is given by
		\begin{equation*}
		c(M,N):=\max\bigg\{\frac{|\langle x,y\rangle|}{\|x\|\,\|y\|}\ :\ x\in M\cap\left(M\cap N\right)^{\perp},y\in N\cap\left(M\cap N\right)^{\perp},\ x,y\neq 0\bigg\}
		\end{equation*}
		and the cosine of the \emph{minimal angle} between $M$ and $N$ is defined as
		\begin{equation*}
		c_0(M,N):=\max\left\{\frac{|\langle x,y\rangle|}{\|x\|\,\|y\|}\ :\ x\in M,y\in N,\ x,y\neq 0\right\},
		\end{equation*}
		where we set $\max(\emptyset):=0$.
	\end{definition}
	Both definitions agree if $M\cap N=\{0\}$. However, they are different in general. We state a few important properties in order to work with these quantities.
	\begin{proposition}\label[proposition]{propositionAngle}
		The following statements are true for all subspaces $M,N\subset\real^d$:
		\begin{enumerate}
			\item $0\leq c(M,N)\leq c_0(M,N)\leq 1$,
			\item $c(M,N)<1$,
			\item $c_0(M,N)<1\ \iff\ M\cap N=\{0\}$,
			\item $c(M,N)=c(N,M)$ and $c_0(M,N)=c_0(N,M)$,
			\item $c(M,N)=c\left(M^{\perp},N^{\perp}\right)$,
			\item $c_0(M,N)=\|P_M P_N\|$,
			\item $c(M,N)=\|P_M P_N-P_{M\cap N}\|$.
		\end{enumerate}
	\end{proposition}
	One can easily check that $P_M P_N$ is the orthogonal projection onto $M\cap N$ if, and only if, $P_M$ and $P_N$ commute. Nevertheless, if they do not commute, it is still possible to describe $P_{M\cap N}$ via $P_M$ and $P_N$ through the method of alternating projections, which is due to von Neumann \cite{Neumann1950}.
	\begin{theorem}\label[theorem]{theoremAlternatingProjections}
		For each two subspaces $M$ and $N$, the method of alternating projections converges:
		\begin{equation*}
		\lim_{k\to\infty}\|(P_M P_N)^k-P_{M\cap N}\|=0.
		\end{equation*}
	\end{theorem}
	A discussion on the speed of convergence can be found in \cite{Deutsch1995}. The following estimate will be enough for our purposes.
	\begin{proposition}\label[proposition]{propositionAlternatingProjectionsEstimate}
		For each two subspaces $M$ and $N$, it holds
		\begin{equation*}
		\forall k:\, \|(P_M P_N)^k-P_{M\cap N}\|\leq c(M,N)^{2k-1}.
		\end{equation*}
	\end{proposition}
	Utilizing the method of alternating projections, we can relate the distance of two intersections to the distance of intersecting subspaces.
	\begin{proposition}\label[proposition]{propositionIntersectionEstimate}
		Let $M,N\subset\real^d$ be two subspaces, and set $\delta:=c_0\left(M^{\perp},N^{\perp}\right)$.\\
		For all subspaces $M',N'\subset\real^d$ with
		\begin{equation*}
		d(M',M)+d(N',N)\leq \frac{1-\delta}{2},
		\end{equation*}
		we have
		\begin{equation*}
		d(M'\cap N',M\cap N)\leq \delta^{2k-1}+\left(\frac{1+\delta}{2}\right)^{2k-1}+k\,\left(d(M',M)+d(N',N)\right)
		\end{equation*}
		with arbitrary $k\in\natural$.
	\end{proposition}
	\begin{proof}
		Assume $M,N$, $\delta$ and $M',N'$ as above. Using the method of alternating projections, we estimate for arbitrary $k\in\natural$:
		\begin{align*}
		\|P_{M'\cap N'}-P_{M\cap N}\|&\leq\|P_{M'\cap N'}-\left(P_{M'}P_{N'}\right)^k\|+\|\left(P_{M'}P_{N'}\right)^k-\left(P_{M}P_{N}\right)^k\|\\
		&\hspace{1em}+\|\left(P_{M}P_{N}\right)^k-P_{M\cap N}\|\\
		&\leq c(M',N')^{2k-1}+\|\left(P_{M'}P_{N'}\right)^k-\left(P_{M}P_{N}\right)^k\|\\
		&\hspace{1em}+c(M,N)^{2k-1}.
		\end{align*}
		Since the minimal angle depends continuously on its subspaces, we have
		\begin{align*}
		c(M',N')&=c\left(\left(M'\right)^{\perp},\left(N'\right)^{\perp}\right)\\
		&\leq c_0\left(\left(M'\right)^{\perp},\left(N'\right)^{\perp}\right)\\
		&= \|P_{\left(M'\right)^{\perp}}P_{\left(N'\right)^{\perp}}\|\\
		&\leq \|P_{\left(M'\right)^{\perp}}P_{\left(N'\right)^{\perp}}-P_{M^{\perp}}P_{\left(N'\right)^{\perp}}\|+\|P_{M^{\perp}}P_{\left(N'\right)^{\perp}}-P_{M^{\perp}}P_{N^{\perp}}\|\\
		&\hspace{1em}+\|P_{M^{\perp}}P_{N^{\perp}}\|\\
		&\leq \|P_{\left(M'\right)^{\perp}}-P_{M^{\perp}}\|+\|P_{\left(N'\right)^{\perp}}-P_{N^{\perp}}\|+\|P_{M^{\perp}}P_{N^{\perp}}\|\\
		&=\|P_{M'}-P_{M}\|+\|P_{N'}-P_{N}\|+\delta\\
		&\leq \frac{1+\delta}{2}.
		\end{align*}
		For the middle summand in the estimate of $\|P_{M'\cap N'}-P_{M\cap N}\|$, we deduce
		\begin{align*}
		&\|\left(P_{M'}P_{N'}\right)^k-\left(P_{M}P_{N}\right)^k\|\\
		&\hspace{1em}\leq\sum_{l=0}^{k-1}\|(P_MP_N)^l(P_{M'}P_{N'})^{k-l}-(P_MP_N)^lP_MP_{N'}(P_{M'}P_{N'})^{k-(l+1)}\|\\
		&\hspace{2em}+\|(P_MP_{N})^lP_MP_{N'}(P_{M'}P_{N'})^{k-(l+1)}-(P_MP_N)^{l+1}(P_{M'}P_{N'})^{k-(l+1)}\|\\
		&\hspace{1em}\leq \sum_{l=0}^{k-1}\|P_{M'}-P_{M}\|+\|P_{N'}-P_{N}\|\\
		&\hspace{1em}=k\,\left(\|P_{M'}-P_{M}\|+\|P_{N'}-P_{N}\|\right).
		\end{align*}
		For the last summand, we remark
		\begin{equation*}
		c(M,N)=c\left(M^{\perp},N^{\perp}\right)\leq c_0\left(M^{\perp},N^{\perp}\right)=\delta.
		\end{equation*}
		Combining the above yields the desired estimate.
	\end{proof}
	Now, assume we are given two converging sequences of subspaces $\left(M_t\right)_{t\in\mathbb{T}}$ and $\left(N_t\right)_{t\in\mathbb{T}}$ with transversal\footnote{Two subspaces $M$ and $N$ are called \textit{transversal} if $M+N=\real^d$. Since $(M+N)^{\perp}=M^\perp\cap N^\perp$, transversality is equivalent to $c_0\left(M^{\perp},N^{\perp}\right)<1$.} limits $M$ and $N$. As an immediate consequence of \cref{propositionIntersectionEstimate} with the right choice of $k=k(t)$, we see that the sequence of intersections $\left(M_t\cap N_t\right)_{t\in\mathbb{T}}$ converges to the intersection of the limits $M\cap N$. Moreover, we show that the speed of convergence on an exponential scale can be preserved in a uniform manner.
	\begin{corollary}\label[corollary]{corollaryConvergenceRateIntersections}
		Let $M,N\subset\real^d$ be two transversal subspaces. Moreover, assume $\left(\mathcal{M}_{t}\right)_{t\in\mathbb{T}}$ and $\left(\mathcal{N}_{t}\right)_{t\in\mathbb{T}}$ are two sequences of collections of subspaces that converge to $M$, resp. $N$, exponentially fast:
		\begin{equation*}
		\lambda_M:=\lambda\left(\sup_{M'\in\mathcal{M}_{t}}d(M',M)\right)<0\hspace{1em}\text{and}\hspace{1em}\lambda_N:=\lambda\left(\sup_{N'\in\mathcal{N}_{t}}d(N',N)\right)<0.
		\end{equation*}
		Then, 
		\begin{equation*}
		\overline{\lambda}\left(\sup_{M'\in\mathcal{M}_{t_1}}\, \sup_{N'\in\mathcal{N}_{t_2}}d(M'\cap N',M\cap N)\right)\leq\max\left(\lambda_M,\lambda_N\right).
		\end{equation*}
	\end{corollary}
	\begin{proof}
		Let $\delta:=c_0\left(M^{\perp},N^{\perp}\right)<1$. Since we have $\lambda_M,\lambda_N<0$ (exp. decay of distances), there is $T>0$ with
		\begin{equation*}
		\sup_{M'\in\mathcal{M}_{t_1}}\, \sup_{N'\in\mathcal{N}_{t_2}}d(M',M)+d(N',N)\leq \frac{1-\delta}{2}
		\end{equation*}
		for all $t_1,t_2\geq T$. Invoking \cref{propositionIntersectionEstimate}, we get
		\begin{align*}
		&\sup_{M'\in\mathcal{M}_{t_1}}\, \sup_{N'\in\mathcal{N}_{t_2}}d(M'\cap N',M\cap N)\\
		&\hspace{1em}\leq \delta^{2k-1}+\left(\frac{1+\delta}{2}\right)^{2k-1}+k\,\left(\sup_{M'\in\mathcal{M}_{t_1}}d(M',M)+\sup_{N'\in\mathcal{N}_{t_2}}d(N',N)\right)
		\end{align*}
		with arbitrary $k\in\natural$. With  $k=k(t_1,t_2):=\left\lceil\min(t_1,t_2)^2\right\rceil$ and by means of \cref{propositionExtendedLyapunovIndexProperties} and \cref{exampleExtendedLyapunovIndex} we compute
		\begin{align*}
		&\overline{\lambda}\left(\sup_{M'\in\mathcal{M}_{t_1}}\, \sup_{N'\in\mathcal{N}_{t_2}}d(M'\cap N',M\cap N)\right)\\
		&\hspace{1em}\leq\max\Bigg(\overline{\lambda}\left(\delta^{2k(t_1,t_2)-1}\right),\,\overline{\lambda}\left(\left(\frac{1+\delta}{2}\right)^{2k(t_1,t_2)-1}\right),\\
		&\hspace{3em}\overline{\lambda}(k(t_1,t_2))+\max\left(\overline{\lambda}\left(\sup_{M'\in\mathcal{M}_{t_1}}d(M',M)\right),\,\overline{\lambda}\left(\sup_{N'\in\mathcal{N}_{t_2}}d(N',N)\right)\right)\Bigg)\\
		&\hspace{1em}= \max\left(\lambda_M,\lambda_N\right).
		\end{align*}
	\end{proof}

	\subsection{Singular Value Decomposition}\label{subsectionSVD}
	We assume degeneracies $d_1+\dots+d_p=d$ with $d_i\geq 1$ to be given. The case $p=d$ is called \emph{nondegenerate}. Moreover, the standard basis of $\real^d$ is denoted by
	\begin{equation*}
	(e):=\left(e_{1_1},e_{1_2},\dots,e_{1_{d_1}},e_{2_1},\dots,e_{2_{d_2}},\dots\dots,e_{p_1},\dots,e_{p_{d_p}}\right).
	\end{equation*}
	In the nondegenerate case, we drop the subindex, i.e. $(e)=(e_1,\dots,e_d)$. Both cases can be translated into each other via $e_{i_k}=e_{d_1+\dots+d_{i-1}+k}$. To further shorten notation, we write $(Ae)$ for the $d$-tuple of vectors we get from applying a linear map $A$ to each vector of $(e)$.
	\begin{definition}\label[definition]{definitionSVD}
		Let $A\in\real^{d\times d}$. The \emph{singular value decomposition (SVD)} of $A$ is given by
		\begin{equation*}
		A=U\Sigma V^T,
		\end{equation*}
		where
		\begin{equation*}
		\Sigma=\textnormal{diag}\left(\sigma_{1_1},\dots,\sigma_{p_{d_p}}\right)
		\end{equation*}
		is the diagonal matrix of \emph{singular values} $\sigma_{i_k}\geq 0$ and $U,V\in \textnormal{O}(d,\real)$ are orthogonal matrices. The columns $(u):=(Ue)$ of $U$ are called \emph{left singular vectors} and the columns $(v):=(Ve)$ of $V$ are called \emph{right singular vectors}.
	\end{definition}
	A connection between left and right singular vectors is established via
	\begin{equation*}
	A v_{i_k}=\sigma_{i_k} u_{i_k}.
	\end{equation*}\par
	In general, the SVD is not unique. Given $A\in \textnormal{Gl}(d,\real)$ we settle for the following ordering:
	\begin{equation}\label{equationSingularValuesOrdering}
	\sigma_{1_1}\geq\dots\geq\sigma_{p_{d_p}}>0.
	\end{equation}
	Later on, every group of singular values will correspond to a different LE. Hence, the inequalities between $\sigma_{i_{d_i}}$ and $\sigma_{(i+1)_1}$ will eventually be strict. In that case, the spaces spanned by singular vectors of one group, i.e. $\textnormal{span}\left(u_{i_1},\dots,u_{i_{d_i}}\right)$ and  $\textnormal{span}\left(v_{i_1},\dots,v_{i_{d_i}}\right)$, are uniquely determined independent of our choice of SVD with \cref{equationSingularValuesOrdering}.\par
	A SVD $\hat{U}\hat{\Sigma} \hat{V}^T$ for the inverse of $A$ is obtained by inverting $A=U\Sigma V^T$ and, heeding \cref{equationSingularValuesOrdering}, reversing the order of singular values and vectors. In other words, a SVD for the inverse is given by $\left(\hat{\sigma}\right)=\left(\frac{1}{\sigma}\right)^r$, $\left(\hat{u}\right)=\left(v\right)^r$, and $\left(\hat{v}\right)=\left(u\right)^r$ with $(.)^r$ being the tuple in reversed order.\par	
	For convenience sake, we denote the smallest and largest singular value in each group by
	\begin{equation*} 
	\sigma_i^{\min}:=\min_{k=1,\dots,d_i}\sigma_{i_k}\hspace{1em}\text{and}\hspace{1em}\sigma_i^{\max}:=\max_{k=1,\dots,d_i}\sigma_{i_k}.
	\end{equation*}

	\subsection{Gram-Schmidt Procedure}\label{subsectionGramSchmidt}
	We define the Gram-Schmidt procedure for subspaces. To this end, let $U_1\oplus\dots\oplus U_p$ be a decomposition of $\real^d$ into subspaces of dimension $\dim U_i=d_i$. Inductively, set
	\begin{equation*}
	F_i:=\bigoplus_{j=1}^{i}U_j\cap \left(\bigoplus_{j=1}^{i-1}U_j\right)^{\perp}
	\end{equation*}
	for $i=1,\dots,p$. Then, $\real^d=F_1\oplus\dots\oplus F_p$ is a decomposition with $\dim F_i=d_i$, $F_i\perp F_j$ for $i\neq j$, and with 
	\begin{equation*}
	\bigoplus_{j=1}^{i}U_j=\bigoplus_{j=1}^{i}F_j
	\end{equation*}
	for all $i$. Actually, the outcome only depends on the filtration
	\begin{equation*}
	\{0\}\subset\overline{U}_1\subset\overline{U}_2\subset\dots\subset\overline{U}_p=\real^d
	\end{equation*}
	given by 
	\begin{equation*}
	\overline{U}_i:=\bigoplus_{j=1}^{i}U_j.
	\end{equation*}\par
	In later scenarios the above spaces are spanned by groups of vectors. Thus, for a given basis $(b)$, set $U_i^{(b)}$ as the span of $b_{i_1},\dots,b_{i_{d_i}}$. From $U_i^{(b)}$ we get $\overline{U}_i^{(b)}$ and $F_i^{(b)}$. The associated orthogonal projection onto $F_i^{(b)}$ will be denoted by
	\begin{equation*}
	P_i^{(b)}:=P_{F_i^{(b)}}.
	\end{equation*}
	It follows that
	\begin{equation*}
	\overline{P}_i^{(b)}:=\sum_{j=1}^i P_j^{(b)}
	\end{equation*}
	is the orthogonal projection onto $\overline{U}_i^{(b)}$. Another consequence of our notation is the relation
	\begin{equation*}
	A\left(\overline{U}_i^{(b)}\right)=\overline{U}_i^{(Ab)}
	\end{equation*}
	for an invertible linear map $A$.

	\subsection{Admissibility}\label{subsectionAdmissibility}
	Ultimately, the MET provides an asymptotic link between singular vectors (resp. singular values) and Oseledets spaces (resp. LEs). Hence, in order to investigate how a tuple of vectors evolves under subsequent application of linear maps and the Gram-Schmidt procedure, we relate it to singular vectors. That relation is represented by a single parameter $\delta$. It describes how strong the corresponding filtrations are correlated. Here, a value of $0$ means no correlation and a value of $1$ implies equality. Thus, we call tuples that have a certain level of correlation admissible. A special task will be to understand how many tuples are at least $\delta$-admissible. For this purpose, we denote by $\mu$ the Lebesgue-measure for the respective dimension.\par
	\begin{definition}\label[definition]{definitionAdmissible}
		Let $0<\delta\leq 1$ and a basis $(c)$ of $\real^d$ be given. A $d$-tuple $(b)$ is called $\delta$-\emph{admissible} with respect to $(c)$ if it is linearly independent and
		\begin{equation*}
		\forall\, i<p:\, d\left(\overline{U}_i^{(b)},\overline{U}_i^{(c)}\right)^2\leq 1-\delta^2.
		\end{equation*}
		We denote the set of all $\delta$-admissible tuples by $\mathcal{A}d^{(c)}(\delta)$ and the set of all tuples that are admissible for some $\delta>0$ by $\mathcal{A}d^{(c)}$.
	\end{definition}
	As admissibility is described by distances of filtration spaces, we are allowed to interchange the involved tuples with their Gram-Schmidt bases. So, let us assume $(c)$ to be an ONB from now on. Moreover, the invariance of distances under orthogonal transformations implies that $\delta$-admissibility of $(b)$ w.r.t. $(c)$ is equivalent to $\delta$-admissibility of $(Vb)$ w.r.t. $(Vc)$ for all $V\in \textnormal{O}(d,\real)$. Hence, $V^d\left(\mathcal{A}d^{(c)}(\delta)\right)$ and $\mathcal{A}d^{(Vc)}(\delta)$ coincide.\par 
	Next, let us proceed with an alternative characterization of admissibility.
	\begin{lemma}\label[lemma]{lemmaAdmissibleAlternateForm}
		A basis $(b)$ is $\delta$-admissible w.r.t. $(c)$ if, and only if, for all $i<p$ and $x\in\overline{U}_i^{(b)}$ with $\|x\|=1$, we have
		\begin{equation*}
		\sum_{j=1}^{i}\sum_k |\langle x,c_{j_k}\rangle|^2\geq\delta^2.
		\end{equation*}
	\end{lemma}
	\begin{proof}
		We reformulate the distance between filtration spaces as follows:
		\begin{alignat*}{2}
		\left\|\left(I-\overline{P}_i^{(c)}\right)\overline{P}_i^{(b)}\right\|^2&=\max_{\|x\|=1}\left\|\left(I-\overline{P}_i^{(c)}\right)\overline{P}_i^{(b)}x\right\|^2&&=\max_{\substack{x\in\overline{\mathcal{U}}_i^{(b)} \\ \|x\|=1}}\left\|\left(I-\overline{P}_i^{(c)}\right)x\right\|^2\\
		&=1-\min_{\substack{x\in\overline{\mathcal{U}}_i^{(b)} \\ \|x\|=1}}\left\|\overline{P}_i^{(c)}x\right\|^2&&=1-\min_{\substack{x\in\overline{\mathcal{U}}_i^{(b)} \\ \|x\|=1}}\ \sum_{j=1}^{i}\sum_k |\langle x,c_{j_k}\rangle|^2.
		\end{alignat*}
	\end{proof}
	Now, we are able to relate the evolution of a tuple under a linear map to singular vectors. As it turns out, the relation is sensitive to the admissibility parameter. In fact, being able to control the following estimate was a major reason to introduce the concept of admissibility.
	\begin{proposition}\label[proposition]{propositionAdmissiblePropagation}
		Let $A=U\Sigma V^T$ be invertible and $0<\delta\leq 1$. For all $(b)\in \mathcal{A}d^{(v)}(\delta)$, it holds
		\begin{equation*}
		\forall i:\, d\left(\overline{U}_i^{(Ab)},\overline{U}_i^{(u)}\right)\leq\frac{1}{\delta}\,\frac{\sigma_{i+1}^{\max}}{\sigma_i^{\min}}.
		\end{equation*}
	\end{proposition}
	\begin{proof}
		First, express $x\in\real^d$ using right singular vectors:
		\begin{equation*}
		x=\sum_{j_k}\langle x,v_{j_k}\rangle v_{j_k}.
		\end{equation*}
		Applying the linear map $A=U\Sigma V^T$, we get
		\begin{equation*}
		Ax=\sum_{j_k}\langle x,v_{j_k}\rangle\sigma_{j_k} u_{j_k}\hspace{1em}
		\Rightarrow\hspace{1em} \|Ax\|^2=\sum_{j_k}|\langle x,v_{j_k}\rangle|^2\sigma_{j_k}^2.
		\end{equation*}
		For $x\in\overline{U}_i^{(b)}$ with $\|x\|=1$, this means
		\begin{equation*}
		\|Ax\|^2\geq \sum_{j=1}^i \sum_k |\langle x,v_{j_k}\rangle|^2\sigma_{j_k}^2\geq \left(\sigma_{i}^{\min}\right)^2\sum_{j=1}^i \sum_k |\langle x,v_{j_k}\rangle|^2\geq\delta^2\left(\sigma_{i}^{\min}\right)^2
		\end{equation*}
		by admissibility of $(b)$. Moreover, the following holds for $x\in\real^d$ with $\|x\|=1$:
		\begin{equation*}
		\left\|\left(I-\overline{P}_i^{(u)}\right)Ax\right\|^2= \sum_{j>i}\sum_k |\langle x,v_{j_k}\rangle|^2\sigma_{j_k}^2\leq \left(\sigma_{i+1}^{\max}\right)^2.
		\end{equation*}
		Now, we compute:
		\begin{alignat*}{2}
		d\left(\overline{U}_i^{(Ab)},\overline{U}_i^{(u)}\right)&=\left\|\left(I-\overline{P}_i^{(u)}\right)\overline{P}_i^{(Ab)}\right\|&&=\max_{y\in\overline{U}_i^{(Ab)}\setminus\{0\}}\frac{\left\|\left(I-\overline{P}_i^{(u)}\right)y\right\|}{\|y\|}\\
		&=\max_{x\in\overline{U}_i^{(b)}\setminus\{0\}}\frac{\left\|\left(I-\overline{P}_i^{(u)}\right)A x\right\|}{\|A x\|}&&=\max_{\substack{x\in\overline{U}_i^{(b)} \\ \|x\|=1}}\frac{\left\|\left(I-\overline{P}_i^{(u)}\right)A x\right\|}{\|Ax\|}\\
		&\leq\frac{1}{\delta}\,\frac{\sigma_{i+1}^{\max}}{\sigma_i^{\min}}.
		\end{alignat*}
	\end{proof}
	The above proposition describes behavior only of admissible tuples. However, it turns out that almost all tuples are admissible. Indeed, for admissibility to be generic, the complement of the open set
	\begin{equation*}
	\mathcal{A}d^{(c)}=\left\{(b)\text{ basis }|\ \forall i:\, d\left(\overline{U}_i^{(b)},\overline{U}_i^{(c)}\right)<1\right\}\subset\left(\real^d\right)^d
	\end{equation*}
	must be a set of measure zero. Using \cref{propositionDistance}, we can rewrite the condition as follows:
	\begin{equation*}
	d\left(\overline{U}_i^{(b)},\overline{U}_i^{(c)}\right)<1\iff\overline{U}_i^{(b)}\oplus\left(\overline{U}_i^{(c)}\right)^{\perp}=\real^d.
	\end{equation*}
	Since $(c)$ is an ONB, we yet have another equivalent formulation on the level of basis vectors:
	\begin{equation*}
	d\left(\overline{U}_i^{(b)},\overline{U}_i^{(c)}\right)<1\iff\det\left(b_{1_1},\dots,b_{i_{d_i}},c_{(i+1)_1},\dots,c_{p_{d_p}}\right) \neq 0.
	\end{equation*}
	This form easily reveals the following:
	\begin{proposition}\label[proposition]{propositionAdmissibleGeneric}
		The set of nonadmissible tuples $\left(\real^d\right)^d\setminus\mathcal{A}d^{(c)}$ has Lebesgue-measure zero.
	\end{proposition}
	\begin{proof}
		In the above expression write vectors of $(b)$ as coefficients in terms of $(c)$. Now, the claim is a direct consequence of the fact that $\det^{-1}(0)\subset\real^{k\times k}$ is a subset of measure zero for all $k\geq 1$.
	\end{proof}
	Restricted to a domain of finite measure, the last proposition tells us that the measure of non-$\delta$-admissible tuples converges to zero as $\delta$ goes to zero.
	\begin{corollary}\label[corollary]{corollaryAdmissibleFiniteMeasure}
		For each subset $\mathcal{F}\subset\left(\real^d\right)^d$ of finite Lebesgue-measure, it holds
		\begin{equation*}
		\lim_{\delta\searrow 0}\mu\left(\mathcal{F}\setminus\mathcal{A}d^{(c)}(\delta)\right)=0.
		\end{equation*}
	\end{corollary}
	\begin{proof}
		This is a direct consequence of the previous result and continuity of the Lebesgue measure:
		\begin{equation*}
		\lim_{\delta\searrow 0}\mu\left(\mathcal{F}\setminus\mathcal{A}d^{(c)}(\delta)\right)=\mu\left(\bigcap_{0<\delta\leq 1}\mathcal{F}\setminus\mathcal{A}d^{(c)}(\delta)\right)=\mu\left(\mathcal{F}\setminus\mathcal{A}d^{(c)}\right)=0.
		\end{equation*}
	\end{proof}
	In the second part of Ginelli's algorithm, we need a special domain for initial tuples $(b)$. Namely, we look at
	\begin{equation*}
	b_{i_1},\dots,b_{i_{d_i}}\in\textnormal{span}\left(c_{i_1},\dots,c_{p_{d_p}}\right)=U^{(c)}_i\oplus\dots\oplus U^{(c)}_p.
	\end{equation*}
	Instead of admissibility, it will be enough that $b_{i_1},\dots,b_{i_{d_i}}$ can be extended to an admissible tuple of the form
	\begin{equation*}
	\left(*,\dots,*,b_{i_1},\dots,b_{i_{d_i}},*,\dots,*\right)\in\mathcal{A}d^{(c)}(\delta)
	\end{equation*}
	for each index $i$. The set of all $(b)$ satisfying this extension property will be denoted by $\mathcal{A}d_{\textnormal{ext}}^{(c)}(\delta)$. We write $\mathcal{A}d_{\textnormal{ext}}^{(c)}$ for the union of these sets over $0<\delta\leq 1$.\par 
	As before, one readily checks that $V^d\left(\mathcal{A}d_{\textnormal{ext}}^{(c)}(\delta)\right)=\mathcal{A}d_{\textnormal{ext}}^{(Vc)}(\delta)$ for $V\in \textnormal{O}(d,\real)$. Moreover, we again conclude that almost all tuples satisfy extendable admissibility.
	\begin{proposition}\label[proposition]{propositionAdmissibleExtGeneric}
		The set
		\begin{equation*}
		\left(\left(U^{(c)}_1\oplus\dots\oplus U^{(c)}_p\right)^{d_1}\times\left(U^{(c)}_2\oplus\dots\oplus U^{(c)}_p\right)^{d_2}\times\dots\times\left(U^{(c)}_p\right)^{d_p}\right)\setminus\mathcal{A}d^{(c)}_{\textnormal{ext}}
		\end{equation*}
		has Lebesgue-measure zero.
	\end{proposition}
	\begin{proof}
		For each $i$, we show that the set of tuples
		\begin{equation*}
		\left(b_{i_1},\dots,b_{i_{d_i}}\right)\in\left(U^{(c)}_i\oplus\dots\oplus U^{(c)}_p\right)^{d_i}
		\end{equation*}
		not satisfying the extension property has Lebesgue-measure zero.\par
		The idea is to apply \cref{propositionAdmissibleGeneric} to a reduced setting for fixed $i$. To this end, look at $\real^{d'}$ with degeneracies $d'=d_1'+\dots+d_{p'}'$ given by $d_j':=d_{i-1+j}$ for all $j=1,\dots,p':=p+1-i$, and let $(e')$ be its standard basis. We get 
		\begin{equation*}
		\mu\left(\left(\real^{d'}\right)^{d'}\setminus\mathcal{A}d^{(e')}\right)=0.
		\end{equation*}
		In particular, this implies
		\begin{equation*}
		\mu\left(\left(\real^{d'}\right)^{d_1'}\setminus\left\{\left(b_{1_1}',\dots,b_{1_{d_1'}}'\right)\text{ has admissible extension}\right\}\right)=0.
		\end{equation*}
		Now, we transfer the result from $\real^{d'}$ to $U^{(c)}_i\oplus\dots\oplus U^{(c)}_p$ by identifying $(e')$ with $(c_{i_1},\dots,c_{p_{d_p}})$. As an identification between orthonormal bases, Lebesgue-measure, distance between subspaces, and admissibility are preserved. Hence, for almost all given tuples $\left(b_{i_1},\dots,b_{i_{d_i}}\right)\in\left(U^{(c)}_i\oplus\dots\oplus U^{(c)}_p\right)^{d_i}$, we find $0<\delta\leq 1$ and $g_{(i+1)_1},\dots,g_{p_{d_p}}\in U^{(c)}_i\oplus\dots\oplus U^{(c)}_p$ such that
		\begin{equation*}
		d\left(\textnormal{span}\left(b_{i_1},\dots,b_{i_{d_i}}\right),U^{(c)}_i\right)^2\leq 1-\delta^2
		\end{equation*}
		and
		\begin{equation*}
		\forall\, j>i:\,d\left(\textnormal{span}\left(b_{i_1},\dots,b_{i_{d_i}},g_{(i+1)_1},\dots,g_{j_{d_j}}\right),U^{(c)}_i\oplus\dots\oplus U^{(c)}_j\right)^2\leq 1-\delta^2.
		\end{equation*}
		We can extend such a tuple 
		\begin{equation*}
		\left(b_{i_1},\dots,b_{i_{d_i}},g_{(i+1)_1},\dots,g_{p_{d_p}}\right)
		\end{equation*}
		to a $\delta$-admissible tuple $\left(g\right)$ by setting $g_{j_k}:=c_{j_k}$ for $j<i$. This concludes the proof.
	\end{proof}
	As a consequence, we get the following corollary: 
	\begin{corollary}\label[corollary]{corollaryAdmissibleExtFiniteMeasure}
		Given a subset $\mathcal{F}\subset\left(U^{(c)}_1\oplus\dots\oplus U^{(c)}_p\right)^{d_1}\times\dots\times\left(U^{(c)}_p\right)^{d_p}$ of finite Lebesgue-measure, it holds
		\begin{equation*}
		\lim_{\delta\searrow 0}\mu\left(\mathcal{F}\setminus\mathcal{A}d_{\textnormal{ext}}^{(c)}(\delta)\right)=0.
		\end{equation*}
	\end{corollary}
	In the discrete time convergence proof of Ginelli's algorithm, a more precise measure-estimate on non-$\delta$-admissible tuples will be necessary. However, it will be sufficient to know the case, where $\mathcal{F}$ is a products of balls. The rest of \cref{subsectionAdmissibility} will be devoted to a rather technical derivation of explicit estimates needed only for the proof of \cref{theoremGinelliConvergenceAlmostEverywhere}.
	\begin{proposition}\label[proposition]{propositionAdmissibleMeasureEstimate}
		Let $d>1$. There is a constant $\eta=\eta(d,M)>0$ such that
		\begin{equation*}
		\mu\left(B_d(0,M)^d\setminus\mathcal{A}d^{(c)}(\delta)\right)\leq \eta\delta^{\frac{1}{d-1}}.
		\end{equation*}
	\end{proposition}
	Two lemmata on how to construct admissible tuples will guide us to the above proposition. Since admissible tuples for the nondegenerate case are admissible for all possible degenerate cases, it is enough to find an estimate for the nondegenerate case.
	\begin{lemma}\label[lemma]{lemmaAdmissibleConstruction1}
		Let $(f)$ be an ONB of $\real^d$. Fix $1<i<d$ and $0<\delta_1,\delta_2\leq 1$. If
		\begin{equation*}
		\left\|P_{\textnormal{span}(f_1,\dots,f_{i-1},c_{i+1},\dots,c_d)}f_i\right\|^2\leq 1-\delta_1^2\hspace{1em}\text{and}\hspace{1em}\left\|\overline{P}_{i-1}^{(f)}\left(I-\overline{P}_{i}^{(c)}\right)\right\|^2\leq 1-\delta_2^2,
		\end{equation*}
		then
		\begin{equation*}
		d\left(\overline{U}_{i}^{(f)},\overline{U}_{i}^{(c)}\right)^2\leq 1-(\delta_1\delta_2)^2.
		\end{equation*}
	\end{lemma}
	\begin{proof}
		First, we reduce the problem to the case $i=2$ and $d=3$: There are unit vectors $f_1'\in\textnormal{span}(f_1,\dots,f_{i-1})$ and $c_3'\in\textnormal{span}(c_{i+1},\dots,c_d)$ such that
		\begin{equation*}
		\left\|\overline{P}_{i}^{(f)}\left(I-\overline{P}_{i}^{(c)}\right)\right\|^2=\left\|\overline{P}_{i}^{(f)}c_3'\right\|^2=\left\|\overline{P}_{i-1}^{(f)}c_3'\right\|^2+|\langle f_i,c_3'\rangle |^2=|\langle f_1',c_3'\rangle |^2+|\langle f_2',c_3'\rangle |^2
		\end{equation*}
		with $f_2':=f_i$. Furthermore, the assumptions yield
		\begin{equation*}
		\left\|P_{\textnormal{span}(f_1',c_3')}f_2'\right\|^2\leq\left\|P_{\textnormal{span}(f_1,\dots,f_{i-1},c_{i+1},\dots,c_d)}f_i\right\|^2\leq 1-\delta_1^2
		\end{equation*}
		and
		\begin{equation*}
		|\langle f_1',c_3'\rangle |^2\leq\left\|\overline{P}_{i-1}^{(f)}c_3'\right\|^2\leq\left\|\overline{P}_{i-1}^{(f)}\left(I-\overline{P}_{i}^{(c)}\right)\right\|^2\leq 1-\delta_2^2.
		\end{equation*}
		In particular, $f_1'$, $f_2'$ and $c_3'$ are linearly independent. Thus, the problem reduces to finding the right estimate to
		\begin{equation*}
		d\left(\overline{U}_{2}^{(f')},\overline{U}_{2}^{(c')}\right)^2=\left\|\overline{P}_{2}^{(f')}c_3'\right\|^2=|\langle f_1',c_3'\rangle |^2+|\langle f_2',c_3'\rangle|^2
		\end{equation*}
		inside $\textnormal{span}(f_1',f_2',c_3')\cong\real^3$, where $(f')$ and $(c')$ are some ONBs of $\textnormal{span}(f_1',f_2',c_3')$ extending $(f_1',f_2')$ and $c_3'$.\par
		The case $i=2$ and $d=3$ can be shown by a short calculation. It holds
		\begin{align*}
		\left\|P_{\textnormal{span}(f_1',c_3')}f_2'\right\|^2&=|\langle f_1',f_2'\rangle|^2+\left|\left\langle\frac{c_3'-\langle f_1',c_3'\rangle f_1'}{\|c_3'-\langle f_1',c_3'\rangle f_1'\|},f_2'\right\rangle\right|^2\\
		&=\frac{|\langle c_3',f_2'\rangle |^2}{\|c_3'-\langle f_1',c_3'\rangle f_1'\|^2}\\
		&=\frac{|\langle c_3',f_2'\rangle |^2}{1-|\langle f_1',c_3'\rangle |^2}.
		\end{align*}
		Thus, by our assumptions:
		\begin{equation*}
		|\langle f_2',c_3'\rangle |^2=\left\|P_{\textnormal{span}(f_1',c_3')}f_2'\right\|^2(1-|\langle f_1',c_3'\rangle |^2)\leq (1-\delta_1^2)(1-|\langle f_1',c_3'\rangle |^2).
		\end{equation*}
		We estimate:
		\begin{align*}
		&|\langle f_1',c_3'\rangle |^2+|\langle f_2',c_3'\rangle |^2\\
		&\hspace{1em}\leq|\langle f_1',c_3'\rangle |^2+(1-\delta_1^2)(1-|\langle f_1',c_3'\rangle |^2)\\
		&\hspace{1em}=1-\delta_1^2+\delta_1^2\,|\langle f_1',c_3'\rangle |^2\\
		&\hspace{1em}\leq 1-\delta_1^2+\delta_1^2(1-\delta_2^2)\\
		&\hspace{1em}=1-(\delta_1\delta_2)^2.
		\end{align*}
	\end{proof}
	The previous lemma can be used to give a sufficient condition for a tuple to be $\delta$-admissible.
	\begin{lemma}\label[lemma]{lemmaAdmissibleConstruction2}
		If a basis $(b)$ satisfies
		\begin{equation*}
		\forall\, i<d:\, \left\|P_{\textnormal{span}(f_1,\dots,f_{i-1},c_{i+1},\dots,c_d)}f_i\right\|^2\leq 1-\left(\delta^{\frac{1}{d-1}}\right)^2,
		\end{equation*}
		where $(f):=\mathcal{GS}(b)$, then $(b)$ is $\delta$-admissible.
	\end{lemma}
	\begin{proof}
		We prove the result by induction over $i$ showing that
		\begin{equation*}
		d\left(\overline{U}_{i}^{(b)},\overline{U}_{i}^{(c)}\right)^2=d\left(\overline{U}_{i}^{(f)},\overline{U}_{i}^{(c)}\right)^2\leq 1-\left(\delta^{\frac{i}{d-1}}\right)^2\leq 1-\delta^2.
		\end{equation*}
		For $i=1$, we have
		\begin{equation*}
		d\left(\overline{U}_{1}^{(f)},\overline{U}_{1}^{(c)}\right)^2=\left\|\left(I-\overline{P}_1^{(c)}\right)f_1\right\|^2=\left\|P_{\textnormal{span}(c_{2},\dots,c_d)}f_1\right\|^2\leq 1-\left(\delta^{\frac{1}{d-1}}\right)^2.
		\end{equation*}
		Let $1<i<d$ and assume the induction hypothesis is true for $i-1$, which implies that
		\begin{equation*}
		\left\|\overline{P}_{i-1}^{(f)}\left(I-\overline{P}_{i}^{(c)}\right)\right\|^2\leq\left\|\overline{P}_{i-1}^{(f)}\left(I-\overline{P}_{i-1}^{(c)}\right)\right\|^2=d\left(\overline{U}_{i-1}^{(f)},\overline{U}_{i-1}^{(c)}\right)^2\leq 1-\left(\delta^{\frac{i-1}{d-1}}\right)^2.
		\end{equation*}
		Simply apply \cref{lemmaAdmissibleConstruction1} to close the induction step.
	\end{proof}
	Now, we prove the proposition.
	\begin{proof}[Proof of \textnormal{\cref{propositionAdmissibleMeasureEstimate}}]
		Set $\tilde{\delta}:=\delta^{\frac{1}{d-1}}$ and let
		\begin{equation*}
		\mathcal{N}:=\left\{\left.(b)\in B_d(0,M)^d\ \right|\ \exists i:\, \det(b_1,\dots,b_i,c_{i+1},\dots,c_d)=0\right\}
		\end{equation*}
		be the set of all nonadmissible vector tuples inside $B_d(0,M)^d$. From \cref{propositionAdmissibleGeneric} we know that $\mathcal{N}$ has measure zero. On its complement we define a continuous mapping into the $d$-fold product of spheres:
		\begin{equation*}
		w:B_d(0,M)^d\setminus\mathcal{N}\to\left(S^{d-1}\right)^d
		\end{equation*}
		with components
		\begin{equation*}
		w_i(b_1,\dots,b_d):=\mathcal{GS}_d(b_1,\dots,b_{i-1},c_{i+1},\dots,c_d,c_i),
		\end{equation*}
		where $\mathcal{GS}_d$ is the last component of the Gram-Schmidt procedure. By construction $w_i=w_i(b_1,\dots,b_d)$ is the unique unit-vector orthogonal to
		\begin{equation*}
		\textnormal{span}(b_1,\dots,b_{i-1},c_{i+1},\dots,c_d)
		\end{equation*}
		with $\langle w_i,c_i\rangle>0$, and only depends on the first $i-1$ vectors of $(b)$. $w$ will help us to measure sets of admissible vectors.\par
		The Gram-Schmidt basis of $(b)$ is constructed by setting $f_i:=\frac{b_i'}{\|b_i'\|}$ with $b_i':=\left(I-\overline{P}_{i-1}^{(b)}\right)b_i$. Assuming $|\langle w_i,b_i\rangle|\geq M\tilde{\delta}$, we get
		\begin{align*}
		\left\|P_{\textnormal{span}(f_1,\dots,f_{i-1},c_{i+1},\dots,c_d)}f_i\right\|^2&=\left\|P_{\textnormal{span}(b_1,\dots,b_{i-1},c_{i+1},\dots,c_d)}f_i\right\|^2\\
		&=1-|\langle w_i,f_i\rangle|^2\\
		&=1-\frac{|\langle w_i,b_i'\rangle|^2}{\|b_i'\|^2}\\
		&=1-\frac{|\langle w_i,b_i\rangle|^2}{\|b_i'\|^2}\\
		&\leq 1-\frac{|\langle w_i,b_i\rangle|^2}{\|b_i\|^2}\\
		&\leq 1-\frac{|\langle w_i,b_i\rangle|^2}{M^2}\\
		&\leq 1-\tilde{\delta}^2.
		\end{align*}
		Hence, if $(b)\in B_d(0,M)^d\setminus\mathcal{N}$ satisfies
		\begin{equation*}
		\forall\, i<d:\, |\langle w_i,b_i\rangle |\geq M\tilde{\delta},
		\end{equation*}
		then $(b)$ is $\delta$-admissible by \cref{lemmaAdmissibleConstruction2}. In particular, the subset of all non-$\delta$-admissible tuples is contained in the subset of all $(b)$, which either do not fulfill the above condition or which are elements of the set of measure zero $\mathcal{N}$. Therefore, a measure-estimate on tuples not fulfilling the condition is enough for the claim:\\
		\begin{align*}
		&\mu\left(B_d(0,M)^d\setminus\mathcal{A}d^{(c)}(\delta)\right)\\
		&\hspace{1em}\leq\color{black}\mu\left(\left\{\left.(b)\in B_d(0,M)^d\setminus\mathcal{N}\ \right|\ \exists\,  i<d:\, |\langle w_i,b_i\rangle |<M\tilde{\delta}\right\}\right)\\
		&\hspace{1em}\leq\sum_{i<d}\mu\left(\left\{\left.(b)\in B_d(0,M)^d\setminus\mathcal{N}\ \right|\ |\langle w_i,b_i\rangle |<M\tilde{\delta}\right\}\right)\\
		&\hspace{1em}=\sum_{i<d}\mu\big(\big\{\left.(b)\in B_d(0,M)^d\ \right|\ \det(b_1,\dots,b_{i-1},c_i,\dots,c_d)\neq 0\text{ and }|\langle w_i,b_i\rangle |<M\tilde{\delta}\big\}\big)\\
		&\hspace{1em}=\sum_{i<d}(\mu(B_d(0,M)))^{d-i}\int_{\left\{\left.(b_1,\dots,b_{i-1})\in B_d(0,M)^{i-1}\ \right|\ \det(b_1,\dots,b_{i-1},c_i,\dots,c_d)\neq 0\right\}}\\
		&\hspace{3em}\int_{\left\{b_i\in B_d(0,M)\ :\ |\langle w_i,b_i\rangle |<M\tilde{\delta}\right\}}1\ db_i\ d(b_1,\dots,b_{i-1})\\
		&\hspace{1em}\underset{(\star)}{=}\sum_{i<d}(\mu(B_d(0,M)))^{d-i}\int_{\left\{\left.(b_1,\dots,b_{i-1})\in B_d(0,M)^{i-1}\ \right|\ \det(b_1,\dots,b_{i-1},c_i,\dots,c_d)\neq 0\right\}}\\
		&\hspace{3em}\int_{\left\{b_i\in B_d(0,M)\ :\ |\langle e_1,b_i\rangle |<M\tilde{\delta}\right\}}1\ db_i\ d(b_1,\dots,b_{i-1})\\
		&\hspace{1em}=\sum_{i<d}(\mu(B_d(0,M)))^{d-1}\mu\left(B_d(0,M)\cap\left(\left(-M\tilde{\delta},M\tilde{\delta}\right)\times\real^{d-1}\right)\right)\\
		&\hspace{1em}\leq(d-1)(\mu(B_d(0,M)))^{d-1}(2M)^d\tilde{\delta}.
		\end{align*}
		We used Fubini's theorem to measure components separately. In $(\star)$ we rotated $w_i$ to the first vector of the standard basis. Afterwards, we enlarged $B_d(0,M)\cap\left(\left(-M\tilde{\delta},M\tilde{\delta}\right)\times\real^{d-1}\right)$ to $\left(-M\tilde{\delta},M\tilde{\delta}\right)\times(-M,M)^{d-1}$ for a simple estimate.\par
		Now, setting $\eta:=(d-1)(\mu(B_d(0,M)))^{d-1}(2M)^d$ yields the desired estimate.
	\end{proof}
	A similar estimate will be necessary for non-$\delta$-admissible tuples inside the special domain.
	\begin{proposition}\label[proposition]{propositionAdmissibleExtMeasureEstimate}
		Let $d>1$. There is a constant $\eta=\eta(d,M)>0$ such that
		\begin{equation*}
		\mu\left(B(M)\setminus\mathcal{A}d_{\textnormal{ext}}^{(c)}(\delta)\right)\leq \eta\delta^{\frac{1}{d-1}},
		\end{equation*}
		where $B(M)$ is given by a product of balls of radius M inside the special domain:
		\begin{equation*}
		B_{d}(0,M)^{d_1}\times\dots\times B_{d_p}(0,M)^{d_p}\subset\left(U^{(c)}_1\oplus\dots\oplus U^{(c)}_p\right)^{d_1}\times\dots\times\left(U^{(c)}_p\right)^{d_p}.
		\end{equation*}
	\end{proposition}
	\begin{proof}
		The proof is similar to the one of \cref{propositionAdmissibleExtGeneric}. Again, it is enough to find such a bound for the set of all tuples in
		\begin{equation*}
		B_{d_i+\dots+d_p}(0,M)^{d_i}\subset\left(U^{(c)}_i\oplus\dots\oplus U^{(c)}_p\right)^{d_i}
		\end{equation*}
		that cannot be extended to a $\delta$-admissible tuple.\par
		Using the same identification as before, we reduce the problem to finding such an estimate for the set
		\begin{equation*}
		B_{d'}(0,M)^{d_1'}\setminus\left\{\left(b_{1_1}',\dots,b_{1_{d_1'}}'\right)\text{ has a $\delta$-admissible extension}\right\}.
		\end{equation*}
		\Cref{propositionAdmissibleMeasureEstimate} yields $\eta'$ only depending on $d'$ and $M$ with
		\begin{equation*}
		\mu\left(B_{d'}\left(0,M\right)^{d'}\setminus\mathcal{A}d^{(e')}(\delta)\right)\leq \eta'\delta^{\frac{1}{d'-1}}.
		\end{equation*}
		This implies
		\begin{align*}
		&\mu\left(B_{d'}(0,M)^{d_1'}\setminus\left\{\left(b_{1_1}',\dots,b_{1_{d_1'}}'\right)\text{ has a $\delta$-admissible extension}\right\}\right)\\
		&\hspace{1em}\leq \left(\frac{1}{\textnormal{vol}\left(B_{d'}(0,M)\right)}\right)^{d'-d_1'}\eta'\delta^{\frac{1}{d'-1}}.
		\end{align*}
		Finally, an estimate $\eta$ only depending on $M$ and $d$ is achieved by taking the maximum over estimates for all possible combinations of degeneracies.
	\end{proof}
	\bigskip
	
		\section{Ginelli's Algorithm}\label{sectionGinellisAlgorithm}
	In this section we define a minimalistic setting suitable for both the MET and Ginelli's algorithm.

	\subsection{Setting and Multiplicative Ergodic Theorem}\label{subsectionSetting}
	Since we want to cover as many applications for Ginelli's algorithm as possible, we do not specify a type of state space or system. Instead, we assume a non-empty set $\Omega=\left\{\theta_t\omega_0\ |\ t\in\mathbb{T}\right\}$ to be the abstract orbit of our state of interest $\omega_0$ respective to the flow $(\theta_t)_{t\in\mathbb{T}}$. Here, $\theta_t:\Omega\to\Omega$ represents the time-$t$-flow on our orbit. The flow should satisfy $\theta_0=\textnormal{id}_{\Omega}$ and $\theta_{s+t}=\theta_s\theta_t$. Remaining information of the linear model is encoded in a cocycle $\Phi(t,\omega)$ assigning a timestep $t$ and a state $\omega$ to the linear propagator on tangent space from $\omega$ to $\theta_t\omega$.
	\begin{definition}\label[definition]{definitionCocycle}
		A map $\Phi:\mathbb{T}\times\Omega\to\real^{d\times d}$ is called a \emph{(linear) cocycle (over $\theta$)} if
		\begin{enumerate}
			\item $\Phi(0,\omega)=\textnormal{id}$,
			\item $\Phi(s+t,\omega)=\Phi(s,\theta_t\omega)\Phi(t,\omega)$,
		\end{enumerate}
		for all $s,t\in\mathbb{T}$ and $\omega\in\Omega$.
	\end{definition}
	Since $\mathbb{T}$ is two-sided, every cocycle is pointwise invertible with inverse
	\begin{equation*}
	\Phi(t,\omega)^{-1}=\Phi(-t,\theta_t\omega).
	\end{equation*}\par
	The Multiplicative Ergodic Theorem of Oseledets \cite{Oseledets1968} not only gives us existence of CLVs, but will play a crucial role in our convergence proof. We state a deterministic version found in \cite{Arnold1998}. It assumes that changes during a short timestep do not matter on an exponential scale and, furthermore, that expansion rates of different volumes are well-defined and do not exceed the exponential scale.
	\begin{proposition}[Deterministic MET]\label[proposition]{propositionDeterministicMET}
		Let $\Phi$ be a cocycle satisfying
		\begin{equation*}
		\lambda\left(\sup_{s\in[0,1]\cap\mathbb{T}}\|\Phi(s,\theta_t\omega_0)^{\pm 1}\|\right)\leq 0
		\end{equation*}
		and assume that 
		\begin{equation*}
		\lim_{t\to\infty}\frac{1}{t}\log\|\wedge^i\Phi(t,\omega_0)\|\in\real\cup\{-\infty\}
		\end{equation*}
		exists for all orders of the wedge product of $\Phi(t,\omega_0)$. Then, there exists a Lyapunov spectrum with a corresponding filtration capturing subspaces of different growth rates:
		\begin{enumerate}
			\item The \emph{Lyapunov spectrum} consists of \emph{Lyapunov exponents} (LEs)
			\begin{equation*}
			\infty>\lambda_1>\dots>\lambda_p\geq-\infty,
			\end{equation*}
			which are the distinct limits of singular values, together with degeneracies $d_1+\dots+d_p=d$:
			\begin{equation*}
			\forall i:\, \forall\, k=1,\dots,d_i:\, \lambda_{i}=\lim_{t\to\infty}\frac{1}{t}\log\sigma_{i_k}\left(\Phi(t,\omega_0)\right).
			\end{equation*}
			\item There is a filtration
			\begin{equation*}
			\real^d=V_1\supset \dots\supset V_p\supset V_{p+1}:=\{0\}
			\end{equation*}
			given by subspaces
			\begin{equation*}
			V_i:=\left\{x\in\real^d\ \left|\ \lim_{t\to\infty}\frac{1}{t}\log\|\Phi(t,\omega_0)x\|\leq\lambda_i \right.\right\}.
			\end{equation*}
			Limits in the definition of $V_i$ exist for all $x\in\real^d$ and take values in $\{\lambda_1,\dots,\lambda_p\}$. Moreover, it holds
			\begin{equation*}
			\dim V_i-\dim V_{i+1}=d_i.
			\end{equation*}
		\end{enumerate}
	\end{proposition}
	The proposition only requires one-sided time and an invertible cocycle to provide the Lyapunov spectrum and filtration at state $\omega_0$ of the orbit. However, since we assumed two-sided time, we immediately get the existence of these quantities for all states along the orbit.
	\begin{corollary}\label[corollary]{corollaryMETWholeOrbit}
		In the setting of \textnormal{\cref{propositionDeterministicMET}} Lyapunov spectrum and filtration are defined for all $\omega\in\Omega$. Furthermore, $p(\omega)$, $\lambda_i(\omega)$ and $d_i(\omega)$ are independent of $\omega$, and the filtration changes in a covariant way:
		\begin{equation*}
		\Phi(t,\omega)V_i(\omega)=V_i(\theta_t\omega).
		\end{equation*}
	\end{corollary}
	\begin{proof}
		The first assumption of \cref{propositionDeterministicMET} is trivially satisfied if we replace $\omega_0$ by $\omega=\theta_u\omega_0$. To prove the second assumption, we use the following properties of the wedge product, which can be found in \cite{Arnold1998}:
		\begin{enumerate}
			\item $\|\wedge^iA\|=\sigma_1(A)\dots\sigma_i(A)$,
			\item $\|\wedge^i(AB)\|\leq\|\wedge^iA\|\, \|\wedge^iB\|$,
		\end{enumerate}
		for $A,B\in\real^{d\times d}$. Now, the existence of
		\begin{equation*}
		\lim_{t\to\infty}\frac{1}{t}\log\|\wedge^i\Phi(t,\theta_u\omega_0)\|<\infty
		\end{equation*}
		follows due to the cocycle property:
		\begin{align*}
		&\left(\frac{t+u}{t}\right)\left(\frac{1}{t+u}\log\|\wedge^i\Phi(t+u,\omega_0)\|\right)-\frac{1}{t}\log\|\wedge^i\Phi(u,\omega_0)\|\\
		&\hspace{1em}\leq\frac{1}{t}\log\|\wedge^i\Phi(t,\theta_u\omega_0)\|\\
		&\hspace{1em}\leq\left(\frac{t+u}{t}\right)\left(\frac{1}{t+u}\log\|\wedge^i\Phi(t+u,\omega_0)\|\right)+\frac{1}{t}\log\|\wedge^i\Phi(-u,\theta_u\omega_0)\|.
		\end{align*}
		Thus, the proposition gives us the existence of a Lyapunov spectrum and filtration at state $\omega=\theta_u\omega_0$. In particular, the above shows that limits of singular values for $\omega$ and $\omega_0$ coincide on an exponential scale. Hence, the Lyapunov exponents and their multiplicities are the same for $\omega_0$ and $\omega$. Finally, the identity for filtrations spaces follows from the definition.
	\end{proof}
	Similar statements can be derived for the time-reversed cocycle $\Phi^-(t,\omega):=\Phi(-t,\omega)$ over the time-reversed flow $\theta^-_t:=\theta_{-t}$. We denote its Lyapunov spectrum by $(\lambda_{i}^-,d_i^-)_{i=1,\dots,p^-}$ and the corresponding filtration spaces by $V_i^-(\omega)$.\par
	In order to define a covariant splitting of the tangent space that captures asymptotic growth rates in both forward and backward time, we require additional assumptions on Lyapunov spectra and associated splittings of $\Phi$ and $\Phi^-$:
	\begin{enumerate}
		\item $p=p^-$, $d_i^-=d_{p+1-i}$ and $\lambda_{i}^-=-\lambda_{p+1-i}$,
		\item $V_{i+1}(\omega_0)\cap V_{p+1-i}^-(\omega_0)=\{0\}$.
	\end{enumerate}
	A direct consequence is the finiteness of LEs. For convenience sake, we set $\lambda_0:=\infty$ and $\lambda_{p+1}:=-\infty$.
	
	\begin{proposition}\label[proposition]{propositionOseledetsSpaces}
		Assuming the above relations between the Lyapunov spectra of $\Phi$ and $\Phi^-$, there exists a splitting $\real^d=E_1(\omega)\oplus\dots\oplus E_p(\omega)$ of the tangent space into so-called \emph{Oseledets spaces}
		\begin{equation}\label{equationOseledetsSpacesIntersection}
		E_i(\omega):=V_i(\omega)\cap V_{p+1-i}^-(\omega).
		\end{equation}
		Furthermore, Oseledets spaces can be characterized via
		\begin{equation}\label{equationOseledetsSpacesAsymptotics}
		x\in E_i(\omega)\setminus\{0\}\iff \lim_{t\to\pm\infty}\frac{1}{|t|}\log\|\Phi(t,\omega)x\|=\pm\lambda_i,
		\end{equation}
		are covariant 
		\begin{equation*}
		\Phi(t,\omega)E_i(\omega)=E_i(\theta_t\omega),
		\end{equation*}
		and satisfy $\dim E_i(\omega)=d_i$.
	\end{proposition}
	\begin{proof}
		The proof is purely algebraic and can be found along the lines of the proof of the MET for two-sided time in \cite{Arnold1998}.
	\end{proof}
	
	\begin{figure}[tbhp]
		\centering
		\begin{tikzpicture}[z=10,scale=0.1,>=stealth]			
		\draw[-] (xyz cs:x=10,z=0) -- (xyz cs:x=10,z=86);
		\draw[-] (xyz cs:x=10,z=92) -- (xyz cs:x=10,z=100);
		\draw[-] (xyz cs:y=10,z=0) -- (xyz cs:y=10,z=100);
		\draw[-] (xyz cs:x=10,z=100) -- (xyz cs:z=100);
		\draw[-] (xyz cs:y=10,z=100) -- (xyz cs:y=4.5,z=100);
		\draw[-] (xyz cs:y=1.5,z=100) -- (xyz cs:z=100);
		
		\draw[-] (xyz cs:x=-5,z=0) -- (xyz cs:x=-5,z=14.2);
		\draw[-] (xyz cs:y=-5,z=0) -- (xyz cs:y=-5,z=11.5);
		\draw[dashed] (xyz cs:x=-5,z=14.2) -- (xyz cs:x=-5,z=100);
		\draw[dashed] (xyz cs:y=-5,z=16) -- (xyz cs:y=-5,z=98);
		\draw[dashed] (xyz cs:x=-5,z=100) -- (xyz cs:z=100);
		\draw[dashed] (xyz cs:y=-3.5,z=100) -- (xyz cs:z=100);
		
		\draw[->] (xyz cs:z=0) -- (xyz cs:z=100) node[above] {$\theta_t\omega_0$};
		
		\draw[color=red,->] plot[smooth,domain=0:31] (\x,{1+0.2*exp(0.12*\x)+\x}) node[left] {$\sim e^{t\lambda_1}$};
		
		\draw[color=blue,->] plot[smooth,domain=0:31] ({8*exp(-0.1*\x)+\x},\x) node[right] {$\sim e^{t\lambda_2}$};
		
		\draw[-] (xyz cs:x=-5) -- (xyz cs:x=1.5);
		\draw[-] (xyz cs:x=6.5) -- (xyz cs:x=10) node[below] {$E_2$};
		\draw[-] (xyz cs:y=-5) -- (xyz cs:y=10) node[left] {$E_1$};
		\node at (-20,0) { };
		\node at (50,0) { };
		
		\node[fill,circle,inner sep=1.5pt,label={right:$\omega_0$}] at (0,0) {};
		\end{tikzpicture}
		\caption{diagonal cocycle for dimension $2$}
		\label{figureOseledetsSpaces}
	\end{figure}
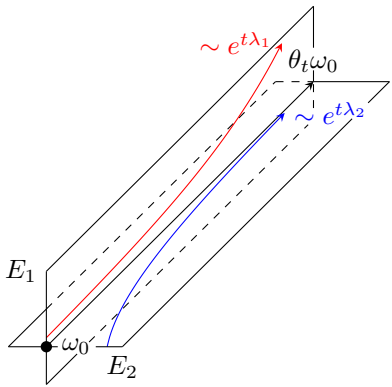
	
	For random dynamical systems satisfying a particular integrability condition, it is shown in \cite{Arnold1998} that the cocycle along almost all orbits of the system admits an Oseledets splitting. Moreover, in an ergodic setting the Lyapunov spectrum coincides for almost all orbits. Therefore, in applications it is often assumed that the underlying system is ergodic at least near an interesting structure.\footnote{See the concept of SRB-measures for attractors \cite{Bowen1975}.} Via CLVs one hopes to better understand the local flow around that structure.
	\begin{definition}\label[definition]{definitionCLVs}
		Normalized basis vectors, which are covariant and chosen subject to the Oseledets splitting for each $\omega\in\Omega$, are called \emph{covariant Lyapunov vectors (CLVs)}.
	\end{definition}
	CLVs represent directions of different asymptotic growth rates\footnote{Aside from asymptotic growth, the angle between CLVs can be used as a measure of hyperbolicity (see, e.g., \cite{Conti2017,Saiki2010,Xu2016,Yang2009}).} by \cref{equationOseledetsSpacesAsymptotics}. However, they are uniquely defined (up to sign) only for nondegenerate Lyapunov spectra.

	\subsection{The Algorithm}\label{subsectionTheAlgorithm}
	The Ginelli algorithm \cite{Ginelli2013, Ginelli2007} computes Oseledets spaces (or CLVs) for a given cocycle by using its asymptotic characterization \cref{equationOseledetsSpacesAsymptotics}. The main idea is that each vector with a nonzero $E_1$-part will approach $E_1$ asymptotically, since its $E_1$-component has the largest exponential growth rate. More abstractly, almost all $(d_1+\dots+d_i)$-dimensional subspaces will align with $E_1\oplus\dots\oplus E_i$, the fastest expanding (or slowest contracting) subspace of the corresponding dimension, in forward time. Reversing time, we are able to extract the slowest expanding (or fastest contracting) subspaces. In particular, almost all $d_i$-dimensional subspaces of $E_1\oplus \dots\oplus E_i$ will align with $E_i$ in backward time.\par
	Taking these traits into consideration, the abstract formalism of Ginelli's algorithm is as follows:\\
	
	\textbf{Ginelli Algorithm (analytical kernel)\footnote{The intended implementation of Ginelli's algorithm has a few more details (see \cite{Ginelli2013, Ginelli2007}). To avoid that all vectors collapse onto the first Oseledets space in forward time, the propagated vectors are frequently orthonormalized via a $QR$-decomposition. Analytically, however, orthonormalizations do not change the filtration of subspaces. Hence, the resulting subspace approximations are the same independent of how often the vectors were corrected.\\
	During phase 1.2 the $R$ matrices are stored and later reused in phase 2. By expressing $(b')$ as coefficients $(\alpha)$ with respect to forward propagated vectors appearing as columns of the $Q$ matrices, it suffices to apply inverses of the $R$ matrices to the coefficient matrix given by $(\alpha)$. In-between propagation steps one normalizes the columns of $(\alpha)$. Similar to the forward phase, propagated subspaces are not changed by the added details. Thus, analytical and numerical approximations coincide if computed with exact precision.}}\\
	\begin{enumerate}
		\item[1.1.] Randomly choose a basis $(b)$ of the tangent space at a past state $\theta_{-t_1}\omega_0$ and propagate it forward until $\omega_0$. If the propagation time $t_1$ is chosen large enough, we expect $\Phi(t_1,\theta_{-t_1}\omega_0)\overline{U}_i^{(b)}=\overline{U}_i^{(\Phi(t_1,\theta_{-t_1}\omega_0)b)}$ to be a good approximation to $E_1(\omega_0)\oplus\dots\oplus E_i(\omega_0)$.
		\item[1.2.] Continue propagating forward until a state $\theta_{t_2}\omega_0$ is reached. This state should be far enough in the future, so that we have a sufficiently good approximation to $E_1\oplus\dots\oplus E_i$ on a long enough timeframe for the second phase.
		\item[2.] For each $i$, randomly choose $d_i$ vectors $b_{i_1}',\dots,b_{i_{d_i}}'$ in $\overline{U}_i^{(\Phi(t_1+t_2,\theta_{-t_1}\omega_0)b)}\approx E_1(\theta_{t_2}\omega_0)\oplus\dots\oplus E_i(\theta_{t_2}\omega_0)$ and propagate them backward until $\omega_0$. The evolved subspace, i.e. $U_i^{(\Phi(-t_2,\theta_{t_2}\omega_0)b')}$, is our approximation to $E_i(\omega_0)$. 
	\end{enumerate}
	Since we propagate vectors forward, we call steps 1.1 and 1.2 \emph{forward phase}, and by the same reasoning step 2 is called \emph{backward phase}.
	
	\begin{figure}[tbhp]
		\centering
		\begin{tikzpicture}[>=stealth]
		\draw[->] (0,0) -- (10,0);
		\draw[<-] (5,-1) -- (10,-1);
		\draw [->] (9,0) arc [start angle=90, end angle=-90, radius=0.5];
		\draw[->] (2.1,0.3) -- (3,0.3);
		\draw[->] (6.4,0.3) -- (7.3,0.3);
		\draw[<-] (6.95,-1.3) -- (7.85,-1.3);
		\draw (1,0.1) -- (1,-0.1);
		\draw (5,0.1) -- (5,-0.1);
		\draw (5,-0.9) -- (5,-1.1);
		\draw (9,0.1) -- (9,-0.1);
		\draw (9,-0.9) -- (9,-1.1);
		\draw (2.55,0.7) node {$\Phi(t_1,\theta_{-t_1}\omega_0)$};
		\draw (6.85,0.7) node {$\Phi(t_2,\omega_0)$};
		\draw (7.4,-1.7) node {$\Phi(-t_2,\theta_{t_2}\omega_0)$};
		\draw (1,-0.5) node {$\theta_{-t_1}\omega_0$};
		\draw (5,-0.5) node {$\omega_0$};
		\draw (9,-0.5) node {$\theta_{t_2}\omega_0$};
		\draw (1,0.3) node {$(b)$};
		\draw (5,0.3) node {$\left(\Phi(t_1,\theta_{-t_1}\omega_0)b\right)$};
		\draw (9,0.3) node {$\left(\Phi(t_1+t_2,\theta_{-t_1}\omega_0)b\right)$};
		\draw (9,-1.3) node {$(b')$};
		\draw (5,-1.3) node {$\left(\Phi(-t_2,\theta_{t_2}\omega_0)b'\right)$};
		\draw (5.2,1.5) node {\Large \textbf{forward phase}};
		\draw (5.2,-2.5) node {\Large \textbf{backward phase}};
		\end{tikzpicture}
		\caption{schematic picture of the Ginelli algorithm}
		\label{figureGinelli}
	\end{figure}
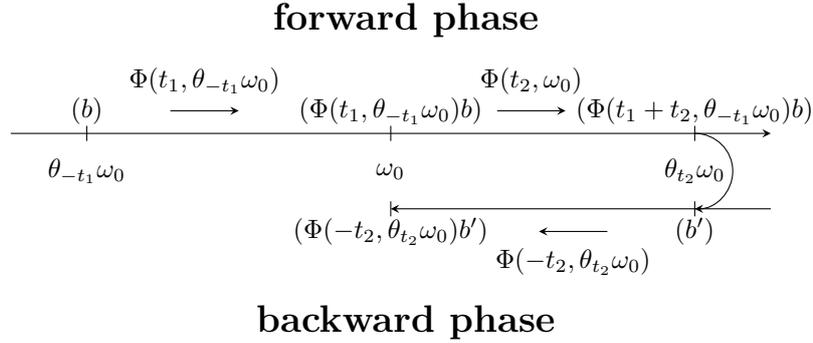
	The asymptotic expansion rate $(\lambda_1+\dots+\lambda_{i})$ of $E_1\oplus\dots\oplus E_i$ is usually computed as a byproduct of the forward phase of the algorithm. Using this information, we can derive the Lyapunov spectrum.\footnote{This concept was already used in 1980 by Benettin \cite{Benettin1980Part1,Benettin1980Part2} to compute the Lyapunov spectrum. Therefore, subsequent applications of the cocycle and orthonormalizations are sometimes called \emph{Benettin steps}.}\par
	In \cref{sectionConvergence} we provide a convergence proof of the whole algorithm as $\min(t_1,t_2)\to\infty$. The speed of convergence turns out to be exponential in relation to the minimum distance of LEs. Furthermore, the kind of convergence differs between discrete and continuous time. The discrete version with $t_1,t_2\in\natural$ converges for almost all initial tuples, whereas the continuous version with $t_1,t_2\in\real_{> 0}$ only converges in measure.

	\subsubsection{Two Examples}
	Next, we present two exemplary cocycles to make the reader familiar with the subtleties of convergence in Ginelli's algorithm.\\
	
	\begin{example}[diagonal cocycle]\label[example]{exampleDiagonalCocycle}
		Assume $\Omega=\{\omega_0\}$ with trivial flow $\theta_t\omega_0=\omega_0$. For given $\lambda_1>\dots>\lambda_p$, define $D:=\textnormal{diag}\left(e^{\lambda_1},\dots,e^{\lambda_p}\right)$. Then, $\Phi(t,\omega_0):=D^t$ is a cocycle and the CLVs (at $\omega_0$) coincide with the standard basis $(e)$ of $\real^d$.\par
		Now, fix a vector $b_1\in\real^d$ with $|\langle b_1,e_1\rangle|>0$. We have
		\begin{equation*}
		\langle D^t b_1,e_i\rangle=\langle b_1,e_i\rangle e^{t\lambda_i}.
		\end{equation*}
		Thus, we compute
		\begin{equation*}
		\left|\left\langle\frac{D^tb_1}{\|D^tb_1\|},e_i\right\rangle\right|^2=\frac{|\langle b_1,e_i\rangle|^2 e^{2t\lambda_i}}{\sum_j |\langle b_1,e_j\rangle|^2 e^{2t\lambda_j}}=\frac{|\langle b_1,e_i\rangle|^2 e^{-2t|\lambda_1-\lambda_i|}}{\sum_j |\langle b_1,e_j\rangle|^2 e^{-2t|\lambda_1-\lambda_j|}}.
		\end{equation*}
		The last nominator takes values between $|\langle b_1,e_1\rangle|^2$ and $\|b_1\|^2$. In particular, it can be treated as a positive constant for the Lyapunov index notation:
		\begin{alignat*}{2}
		\lambda\left(d\left(\overline{U}_1^{(D^tb)},\overline{U}_1^{(e)}\right)\right)&=\frac{1}{2}\,\lambda\left(\left\|\left(I-\overline{P}_1^{(e)}\right)\overline{P}_1^{(D^tb)}\right\|^2\right)&&=\frac{1}{2}\,\lambda\left(\sum_{i\neq 1}\left|\left\langle\frac{D^tb_1}{\|D^tb_1\|},e_i\right\rangle\right|^2\right)\\
		&\leq \max_{i\neq 1} \frac{1}{2}\,\lambda\left(\left|\left\langle\frac{D^tb_1}{\|D^tb_1\|},e_i\right\rangle\right|^2\right)&&\leq \max_{i\neq 1}\,-|\lambda_1-\lambda_i|\\
		&=-|\lambda_1-\lambda_2|.
		\end{alignat*}
		In general, it holds
		\begin{equation*}
		\lambda\left(d\left(\overline{U}_i^{(D^tb)},\overline{U}_i^{(e)}\right)\right)\leq -|\lambda_{i}-\lambda_{i+1}|
		\end{equation*}
		for all tuples $(b)$ that are admissible w.r.t. $(e)$. A more general statement for arbitrary cocycles will be proved in \textnormal{\cref{subsectionForwardPhase}} when analyzing convergence of the forward phase.
	\end{example}
	
	Ginelli's algorithm starts with a random choice of initial vectors to prevent nonadmissible configurations. One such configuration would be the unlikely case where the first vector lies in the second Oseledets space. As Oseledets spaces are covariant, the first vector will stay inside the second Oseledets space when propagated. Consequently, it will not be a good representation of a vector lying in the first Oseledets space. The next example shows that all vectors might be nonadmissible when initiated at a wrong time in the continuous version of Ginelli's algorithm.
	\begin{example}[rotating Oseledets spaces]\label[example]{exampleRotatingOseledetsSpaces}
		Let $\Omega:=S^1\cong\qspace{\real}{\integer}$ be a periodic orbit with homogeneous flow $\theta_t\omega:=\omega+t$. Furthermore, let $R:\real\to \textnormal{SO}(2)$ be the parametrization of $\textnormal{SO}(2)$ by $2\times 2$ rotation matrices
		\begin{equation*}
		R(\omega):=\begin{pmatrix}
		\cos(2\pi\omega) & -\sin(2\pi\omega)\\
		\sin(2\pi\omega) & \cos(2\pi\omega)
		\end{pmatrix},
		\end{equation*}
		so that $R(0)=R(1)=I$ and $R(s+t)=R(s)R(t)$. Moreover, we set $D:=\textnormal{diag}\left(e^{\lambda_1},e^{\lambda_2}\right)$ for some $\lambda_1>\lambda_2$, and define the cocycle to be
		\begin{equation*}
		\Phi(t,\omega):=R(\theta_t\omega)D^tR(-\omega).
		\end{equation*}
		One readily checks that $\Phi(t,\omega)$ indeed is a cocycle over $\theta$.\par		
		Next, we use the characterization of Oseledets spaces via asymptotic growth rates:
		\begin{align*}
		\lim\limits_{t\to\pm\infty}\frac{1}{t}\log\left\|\Phi(t,\omega)\, R(\omega)\begin{pmatrix} x_1 \\ x_2 \end{pmatrix}\right\|&=
		\begin{cases}
		\lambda_1 &x_1\neq 0\text{ and }x_2=0 \\
		\lambda_2 &x_1=0\text{ and }x_2\neq 0
		\end{cases}\\
		\Rightarrow\hspace{1em}E_1(\omega)=\textnormal{span}\left(R(\omega)\begin{pmatrix} 1 \\ 0 \end{pmatrix}\right)\hspace{1em}&\text{and}\hspace{1em}E_2(\omega)=\textnormal{span}\left(R(\omega)\begin{pmatrix} 0 \\ 1 \end{pmatrix}\right).
		\end{align*}
		In particular, both Oseledets spaces are rotating uniformly with $\omega$. Hence, for every fixed vector $b_1\in\real^2$ and $T>0$, we find $t_1\in\real_{>0}$ bigger than $T$ with $b_1\in E_2(\theta_{-t_1}\omega)$. This implies that the continuous version of Ginelli's algorithm does not converge for all fixed choices of $b_1$. Instead, it is shown later that the continuous version converges in measure, i.e. if $b_1$ is chosen randomly.\par
		In the discrete case, however, the set $\bigcup_{t_1\in\natural}E_2(\theta_{-t_1}\omega)$ has Lebesgue-measure zero indicating that the above problem occurs only on a set of measure zero. In fact, we will show convergence for almost all initial tuples in the discrete time case.
	\end{example}
	Setting $D=\textnormal{diag}(e^{\lambda_1},e^{\lambda_1})$ in the previous example yields a trivial Oseledets space $E_1(\omega)=\real^2$ with inner rotation. In general, Oseledets spaces can have complicated internal dynamics that prevent single propagated vectors from converging. Additionally, we already remarked that CLVs are not uniquely defined in the presence of degeneracies. Therefore, objects of interest should not be the propagated vectors themselves, but rather the spaces spanned by them subject to degeneracies.\footnote{In practice, degeneracies can be derived from growth rates of propagated vectors during the forward phase. Moreover, they might be forced by symmetries (e.g. in equivariant systems), whereas, for some classes of systems degenerate scenarios are the exception \cite{Arnold1999}.}\par
	As a closing remark for this section, we would like to mention that there are several other recently developed algorithms, see \cite{Froyland2013,Kuptsov2012,Wolfe2007}, some of which can be treated in a similar fashion to Ginelli's algorithm with the tools developed here.
	\bigskip
	
	\section{Convergence of Ginelli's Algorithm}\label{sectionConvergence}
	Finally, we have gathered enough background knowledge to prove convergence of Ginelli's algorithm. During the proof, we will not distinguish between discrete and continuous time until after we have shown convergence in measure for both cases. Most results will be formulated using the Lyapunov index notation, thus, providing us with a direct link to the exponential speed of convergence.\par

\subsection{Convergence Theorems}\label{subsectionConvergenceTheorems}
	Before formulating our new convergence results, we make one simplification that is motivated by the implementation of Ginelli's algorithm and helps us to formulate the theorems in a more compact way. Namely, as the domain of $(b')$ in the backward phase depends on evolved vectors from the forward phase, it will be convenient to identify the backward domain with a time-independent one. To this end, we set $A^{(f)}\in \textnormal{O}(d)$ as the orthogonal transformation sending the standard basis $(e)$ to the Gram-Schmidt basis of evolved vectors from the forward phase, i.e. $(f)=\mathcal{GS}(\Phi(t_1+t_2,\theta_{-t_1}\omega_0)b)$. Note that forward initial vectors $(b)$ need to be linearly independent in order to get a well-defined mapping. By identifying $\real^{d_1+\dots+d_i}$ with $\real^{d_1+\dots+d_i}\times\{0\}\subset\real^d$ we may regard the restriction of $A^{(f)}$ as an identification between time-independent coefficients and time-dependent vectors:
	\begin{align*}
	\real^{d_1+\dots+d_i}&\to\overline{U}_i^{(\Phi(t_1+t_2,\theta_{-t_1}\omega_0)b)}\\
	\alpha_{i_k}&\mapsto b'_{i_k}.
	\end{align*}
	Thus, we use
	\begin{equation*}
	\left(\real^{d_1}\right)^{d_1}\times\left(\real^{d_1+d_2}\right)^{d_2}\times\dots\times\left(\real^{d_1+\dots+d_{p-1}}\right)^{d_{p-1}}\times\left(\real^{d}\right)^{d_p}\subset \left(\real^d\right)^d
	\end{equation*}
	as the domain for coefficient of the backward phase.
	\begin{theorem}[Convergence in measure of Ginelli's algorithm]\label[theorem]{theoremGinelliConvergenceInMeasure}
		For each compact subset
		\begin{equation*}
		\mathcal{K}\subset\left(\real^d\right)^d\times\left(\left(\real^{d_1}\right)^{d_1}\times\left(\real^{d_1+d_2}\right)^{d_2}\times\dots\times\left(\real^{d_1+\dots+d_{p-1}}\right)^{d_{p-1}}\times\left(\real^{d}\right)^{d_p}\right)
		\end{equation*}
		and $\epsilon>0$, it holds
		\begin{align*}
		&\lim_{T\to\infty}\, \sup_{t_1,t_2\geq T}\mu\bigg(\mathcal{K}\setminus\bigg\{\left((b),(\alpha)\right)\in \mathcal{K}\ \big|\ (b)\text{ linearly independent and }\forall i:\\ &\hspace{3em}\frac{1}{\min(t_1,t_2)}\log d\left(U_i^{\left(\Phi(-t_2,\theta_{t_2}\omega_0)b'\right)},E_i(\omega_0)\right)\\
		&\hspace{3em}\leq-\min\left(|\lambda_i-\lambda_{i-1}|,|\lambda_i-\lambda_{i+1}|\right)+\epsilon\\
		&\hspace{3em}\text{with }(b')=\left(A^{\mathcal{GS}(\Phi(t_1+t_2,\theta_{-t_1}\omega_0)b)}\alpha\right)\bigg\}\bigg)\\
		&\hspace{1em}=0.
		\end{align*}
	\end{theorem}
	Compared to the somewhat more involved notation of \cref{theoremGinelliConvergenceInMeasure} the convergence theorem for discrete time can be formulated quite nicely using the Lyapunov index notation.
	\begin{theorem}[Convergence a.e. of Ginelli's algorithm for $\mathbb{T}=\integer$]\label[theorem]{theoremGinelliConvergenceAlmostEverywhere}
		For almost all pairs of tuples $((b),(\alpha))$ 
		\begin{equation*}
		\left(\real^d\right)^d\times\left(\left(\real^{d_1}\right)^{d_1}\times\left(\real^{d_1+d_2}\right)^{d_2}\times\dots\times\left(\real^{d_1+\dots+d_{p-1}}\right)^{d_{p-1}}\times\left(\real^{d}\right)^{d_p}\right),
		\end{equation*}
		$(b)$ is linearly independent and the algorithm converges:
		\begin{equation*}
		\overline{\lambda}\left(d\left(U_i^{\left(\Phi(-n_2,\theta_{n_2}\omega_0)b'\right)},E_i(\omega_0)\right)\right)\leq-\min\left(|\lambda_i-\lambda_{i-1}|,|\lambda_i-\lambda_{i+1}|\right)
		\end{equation*}
		with $(b')=\left(A^{\mathcal{GS}\left(\Phi(n_1+n_2,\theta_{-n_1}\omega)b\right)}\alpha\right)$.
	\end{theorem}
	\Cref{theoremGinelliConvergenceAlmostEverywhere} tells us that, for almost all choices of initial vectors $(b)$ for the forward phase and initial coefficients $(\alpha)$ for the backward phase, the $i$-th output subspace $U_i^{\left(\Phi(-n_2,\theta_{n_2}\omega_0)b'\right)}$ of Ginelli's algorithm converges to $E_i(\omega_0)$ exponentially fast with a rate of $\min\left(|\lambda_i-\lambda_{i-1}|,|\lambda_i-\lambda_{i+1}|\right)$. In particular, the speed of convergence of the whole algorithm is approximately given by \begin{equation*}
		e^{-\min(n_1,n_2)\min_i|\lambda_{i}-\lambda_{i+1}|}.
	\end{equation*}\par
	In applications one usually wants to compute CLVs at more than just one point along a trajectory.\footnote{It is much harder to predict how the rate of convergence changes when switching to another orbit. For example, in the scenario of random dynamical systems as in \cite{Arnold1998} Lyapunov spectrum and Oseledets spaces depend only measurably on $\omega_0$.} In fact, it is feasible to use propagated vectors near $\omega_0$ as approximations to CLVs in Ginelli's algorithm. Thus, it is enough to run the algorithm once. Similar statements on convergence are possible. We only formulate a version for discrete time. 
	\begin{corollary}[Convergence a.e. of Ginelli's algorithm on interval for $\mathbb{T}=\integer$]\label[corollary]{corollaryCLVInterval}
		Let $I\subset\mathbb{T}$ be a bounded interval. For almost all pairs of tuples $((b),(\alpha))$ in
		\begin{equation*}
		\left(\real^d\right)^d\times\left(\left(\real^{d_1}\right)^{d_1}\times\left(\real^{d_1+d_2}\right)^{d_2}\times\dots\times\left(\real^{d_1+\dots+d_{p-1}}\right)^{d_{p-1}}\times\left(\real^{d}\right)^{d_p}\right),
		\end{equation*}
		$(b)$ is linearly independent and the algorithm converges on $I$:
		\begin{equation*}
		\overline{\lambda}\left(\sup_{m\in I}d\left(U_i^{\left(\Phi(-n_2+m,\theta_{n_2}\omega_0)b'\right)},E_i(\theta_m\omega_0)\right)\right)\leq-\min\left(|\lambda_i-\lambda_{i-1}|,|\lambda_i-\lambda_{i+1}|\right)
		\end{equation*}
		with $(b')=\left(A^{\mathcal{GS}\left(\Phi(n_1+n_2,\theta_{-n_1}\omega)b\right)}\alpha\right)$.
	\end{corollary}
	\begin{proof}
		Writing
		\begin{multline*}
		d\left(U_i^{\left(\Phi(-n_2+m,\theta_{n_2}\omega_0)b'\right)},E_i(\theta_m\omega_0)\right)\hspace{1em}\\
		=d\left(\Phi(m,\omega_0)U_i^{\left(\Phi(-n_2,\theta_{n_2}\omega_0)b'\right)},\Phi(m,\omega_0)E_i(\omega_0)\right),
		\end{multline*}
		the claim is a direct consequence of \cref{theoremGinelliConvergenceAlmostEverywhere} and \cref{corollaryInducedTransformationLipschitz}.
	\end{proof}
	In order to prove both theorems, we derive asymptotic characterizations of each phase of Ginelli's algorithm. However, first, we need to understand how singular vectors and Oseledets spaces are connected by invoking the proof of \cref{propositionDeterministicMET} as it can be found in \cite{Arnold1998}.

	\subsection{The Link between Multiplicative Ergodic Theorem and Singular Value Decomposition}\label{subsectionMETSingularVectors}
	Let
	\begin{equation*}
	\Phi(t,\omega_0)=U(t)\Sigma(t) \left(V(t)\right)^T
	\end{equation*}
	be a SVD of the cocycle $\Phi(t,\omega_0)$ for $t\geq 0$, where singular values are ordered as in \cref{equationSingularValuesOrdering}. Using right singular vectors, Arnold shows that the filtration $V_1(t)\supset\dots\supset V_p(t)$ given by
	\begin{equation*}
	V_i(t):=\left(\overline{U}_{i-1}^{(v(t))}\right)^{\perp}
	\end{equation*}
	converges exponentially fast to the filtration $V_1(\omega_0)\supset\dots\supset V_p(\omega_0)$. Distances between filtrations are measured in a special metric. Unraveling the notation, we end up with
	\begin{equation*}
	\forall\, i\neq j:\, \lambda\left(\left\|P_i^{(v(t))}P_j\right\|\right)\leq -|\lambda_{i}-\lambda_{j}|,
	\end{equation*}
	where $P_p+\dots+P_i$ is the orthogonal projection onto $V_i(\omega_0)$ for each $i$.
	\begin{lemma}\label[lemma]{lemmaMETProof}
		It holds
		\begin{equation*}
		\forall i:\, \lambda\left(d\left(\overline{U}_i^{(v(t))},\left(V_{i+1}(\omega_0)\right)^{\perp}\right)\right)\leq-|\lambda_{i}-\lambda_{i+1}|.
		\end{equation*}
	\end{lemma}
	\begin{proof}
		We compute
		\begin{alignat*}{2}
		\lambda\left(d\left(\overline{U}_i^{(v(t))},\left(V_{i+1}(\omega_0)\right)^{\perp}\right)\right)&=\lambda\left(\left\|\overline{P}_i^{(v(t))}P_{V_{i+1}(\omega_0)}\right\|\right)&&\leq\lambda\left(\sum_{\substack{k,j \\ k\leq i<j}}\left\|P_k^{(v(t))}P_j\right\|\right)\\
		&\leq \max_{\substack{k,j \\ k\leq i<j}}\,-|\lambda_k-\lambda_j|&&=-|\lambda_{i}-\lambda_{i+1}|.
		\end{alignat*}
	\end{proof}
	A similar result holds for the time-reversed cocycle $\Phi^-$ with SVD
	\begin{equation*}
	\Phi(-t,\omega_0)=U^-(t)\Sigma^-(t) \left(V^-(t)\right)^T
	\end{equation*}
	for $t\geq 0$, where singular values are ordered as in \cref{equationSingularValuesOrdering}. Note that, for the time-reversed cocycle, we need to consider reversed degeneracies: $d^-_1,\dots,d_p^-$. To distinguish between both types of degeneracies we equip the notation introduced in \cref{subsectionGramSchmidt} with a minus sign following the subindex, whenever we count with respect to reversed degeneracies.
	\begin{lemma}\label[lemma]{lemmaMETProofReversedTime}
		It holds
		\begin{equation*}
		\forall i:\, \lambda\left(d\left(\overline{U}_{i^-}^{(v^-(t))},\left(V^-_{i+1}(\omega_0)\right)^{\perp}\right)\right)\leq-|\lambda^-_{i}-\lambda^-_{i+1}|.
		\end{equation*}
	\end{lemma}
	The algorithm of Ginelli starts by propagating vectors from past to present, i.e. we apply $\Phi(t,\theta_{-t}\omega_0)=\left(\Phi(-t,\omega_0)\right)^{-1}$, and ends with propagating vectors from future to present, i.e. we apply $\Phi(-t,\theta_{t}\omega_0)=\left(\Phi(t,\omega_0)\right)^{-1}$. Thus, it is important to keep track of singular vectors for inverted cocycles as well.
	\begin{lemma}\label[lemma]{lemmaMETProofInversed}
		It holds
		\begin{equation*}
		\forall i:\, \lambda\left(d\left(\overline{U}_{i^-}^{(\hat{u}(t))},V_{p+1-i}(\omega_0)\right)\right)\leq-|\lambda_{p-i}-\lambda_{p+1-i}|.
		\end{equation*}
	\end{lemma}
	\begin{proof}
		This is a consequence of \cref{lemmaMETProof}, since
		\begin{align*}
		d\left(\overline{U}_{i^-}^{(\hat{u}(t))},V_{p+1-i}(\omega_0)\right)&=d\left(\overline{U}_{i^-}^{(v(t))^r},V_{p+1-i}(\omega_0)\right)\\
		&=d\left(\left(\overline{U}_{p-i}^{(v(t))}\right)^{\perp},V_{p+1-i}(\omega_0)\right)\\
		&=d\left(\overline{U}_{p-i}^{(v(t))},\left(V_{p+1-i}(\omega_0)\right)^{\perp}\right).
		\end{align*}
		Here, we used the identity
		\begin{equation*}
		\overline{U}_{i^-}^{(c)^r}=\left(\overline{U}_{p-i}^{(c)}\right)^{\perp},
		\end{equation*}
		which is true for all ONBs $(c)$.
	\end{proof}
	Again, we derive a similar result for reversed time.
	\begin{lemma}\label[lemma]{lemmaMETProofReversedTimeInversed}
		It holds
		\begin{equation*}
		\forall i:\, \lambda\left(d\left(\overline{U}_{i}^{(\hat{u}^-(t))},V^-_{p+1-i}(\omega_0)\right)\right)\leq-|\lambda_{i}-\lambda_{i+1}|.
		\end{equation*}
	\end{lemma}

	\subsection{Forward Phase}\label{subsectionForwardPhase}
	Step 1.1 of Ginelli's algorithm propagates vectors from past to present. It turns out that admissible tuples yield good approximations to $V_{p+1-i}^-(\omega_0)=E_1(\omega_0)\oplus\dots\oplus E_i(\omega_0)$. Moreover, changes of the admissibility parameter on subexponential scales do not influence the exponential speed of convergence of the algorithm. 
	\begin{lemma}\label[lemma]{lemmaProofStep1.1}
		Let $0<\delta(t)<1$ be a sequence with $\lambda\left(\frac{1}{\delta}\right)=0$. We have
		\begin{equation*}
		\lambda\left(\sup_{(b)\in\mathcal{A}d^{(\hat{v}^-(t))}(\delta(t))}d\left(\overline{U}_i^{\left(\Phi(t,\theta_{-t}\omega_0)b\right)}, V_{p+1-i}^-(\omega_0)\right)\right)\leq -|\lambda_{i}-\lambda_{i+1}|.
		\end{equation*}
	\end{lemma}
	\begin{proof}
		First use the triangle inequality, then apply \cref{propositionAdmissiblePropagation} to the map $A=\left(\Phi(-t,\omega_0)\right)^{-1}$, and finally use \cref{lemmaMETProofReversedTimeInversed} to obtain
		\begin{align*}
		&\lambda\left(\sup_{(b)\in\mathcal{A}d^{(\hat{v}^-(t))}(\delta(t))}d\left(\overline{U}_i^{\left(\Phi(t,\theta_{-t}\omega_0)b\right)}, V_{p+1-i}^-(\omega_0)\right)\right)\\
		&\hspace{1em}\leq\max\Bigg( \lambda\left(\sup_{(b)\in\mathcal{A}d^{(\hat{v}^-(t))}(\delta(t))}d\left(\overline{U}_i^{\left(\Phi(t,\theta_{-t}\omega_0)b\right)},\overline{U}_i^{(\hat{u}^-(t))}\right)\right),\\
		&\hspace{3em}\lambda\left(d\left(\overline{U}_i^{(\hat{u}^-(t))}, V_{p+1-i}^-(\omega_0)\right)\right)\Bigg)\\
		&\hspace{1em}\leq\max\left(\lambda\left(\frac{1}{\delta(t)}\,\frac{\left(\hat{\sigma}^-_{i+1}(t)\right)^{\max}}{\left(\hat{\sigma}^-_{i}(t)\right)^{\min}}\right),-|\lambda_{i}-\lambda_{i+1}|\right)\\
		&\hspace{1em}\leq\max\left(\lambda\left(\frac{\left({\sigma}^-_{p+1-i}(t)\right)^{\max}}{\left({\sigma}^-_{p-i}(t)\right)^{\min}}\right),-|\lambda_{i}-\lambda_{i+1}|\right)\\
		&\hspace{1em}=-|\lambda_{i}-\lambda_{i+1}|.
		\end{align*}
	\end{proof}
	To continue using our tools for step 1.2 we need to retain admissibility for tuples propagated in step 1.1.
	\begin{lemma}\label[lemma]{lemmaAdmissibleAfterStep1.1}
		Let $0<\delta(t)<1$ with $\lambda\left(\frac{1}{\delta}\right)=0$. There are $0<\epsilon<1$ and $T>0$ such that admissible tuples in step 1.1 get mapped to admissible tuples for step 1.2, i.e.
		\begin{equation*}
		\left(\Phi(t_1,\theta_{-t_1}\omega_0)\right)^d\left(\mathcal{A}d^{(\hat{v}^-(t_1))}(\delta(t_1))\right)\subset\mathcal{A}d^{(v(t_2))}(\epsilon),
		\end{equation*}
		for all $t_1,t_2\geq T$.
	\end{lemma}
	\begin{proof}
		Choose $0<\epsilon<1$ with 
		\begin{equation*}
		d\left(V_{p+1-i}^-(\omega_0),\left(V_{i+1}(\omega_0)\right)^{\perp}\right)\leq\sqrt{1-\epsilon^2}-2\epsilon.
		\end{equation*}
		This is possible due to \cref{propositionDistance}, since we assumed $V_{p+1-i}^-(\omega_0)\cap V_{i+1}(\omega_0)=\{0\}$. Now, \Cref{lemmaProofStep1.1} gives us the existence of $T_1>0$ such that for all $t_1\geq T_1$ and all $(b)\in\mathcal{A}d^{\left(\hat{v}^-(t_1)\right)}(\delta(t_1))$ it holds
		\begin{equation*}
		d\left(\overline{U}_i^{\left(\Phi(t_1,\theta_{-t_1}\omega_0)b\right)},V_{p+1-i}^-(\omega_0)\right)\leq\epsilon.
		\end{equation*}
		Moreover, \cref{lemmaMETProof} yields $T_2>0$ with
		\begin{equation*}
		d\left(\left(V_{i+1}(\omega_0)\right)^{\perp},\overline{U}_i^{(v(t_2))}\right)\leq\epsilon
		\end{equation*}
		for all $t_2\geq T_2$. Set $T:=\max(T_1,T_2)$ and combine the previous three estimates for
		\begin{align*}
		d\left(\overline{U}_i^{\left(\Phi(t_1,\theta_{-t_1}\omega_0)b\right)},\overline{U}_i^{(v(t_2))}\right)&\leq  d\left(\overline{U}_i^{\left(\Phi(t_1,\theta_{-t_1}\omega_0)b\right)},V_{p+1-i}^-(\omega_0)\right)\\
		&\hspace{1em}+d\left(V_{p+1-i}^-(\omega_0),\left(V_{i+1}(\omega_0)\right)^{\perp}\right)\\
		&\hspace{1em}+d\left(\left(V_{i+1}(\omega_0)\right)^{\perp},\overline{U}_i^{(v(t_2))}\right)\\
		&\leq\sqrt{1-\epsilon^2}.
		\end{align*}
		This concludes the proof.
	\end{proof}
	The following lemma combines step 1.1 and 1.2 into a characterization of the forward phase.
	\begin{lemma}\label[lemma]{lemmaProofStep1.2}
		Let $0<\delta(t)<1$ with $\lambda\left(\frac{1}{\delta}\right)=0$. There is $T>0$ such that
		\begin{equation*}
		\lambda\left(\sup_{t_1\geq T}\, \sup_{(b)\in\mathcal{A}d^{(\hat{v}^-(t_1))}(\delta(t_1))}d\left(\overline{U}_i^{\left(\Phi(t_1+t_2,\theta_{-t_1}\omega_0)b\right)},\overline{U}_i^{(u(t_2))}\right)\right)\leq -|\lambda_{i}-\lambda_{i+1}|
		\end{equation*}
		holds, where the limit of the Lyapunov index is taken with respect to $t_2$.
	\end{lemma}
	\begin{proof}
		Write 
		\begin{equation*}
		\Phi(t_1+t_2,\theta_{-t_1}\omega_0)=\Phi(t_2,\omega_0)\Phi(t_1,\theta_{-t_1}\omega_0).
		\end{equation*}
		By \cref{lemmaAdmissibleAfterStep1.1} we find $T>0$ and $0<\epsilon<1$ such that for all $t_1,t_2\geq T$ and $(b)\in\mathcal{A}d^{(\hat{v}^-(t_1))}(\delta(t_1))$ the tuple $\left(\Phi(t_1,\theta_{-t_1}\omega_0)b\right)$ is $\epsilon$-admissible w.r.t. $v(t_2)$. Now, apply \cref{propositionAdmissiblePropagation} with $A=\Phi(t_2,\omega_0)$ to see that
		\begin{equation*}
		d\left(\overline{U}_i^{\left(\Phi(t_1+t_2,\theta_{-t_1}\omega_0)b\right)},\overline{U}_i^{(u(t_2))}\right)\leq \frac{1}{\epsilon}\,\frac{\sigma_{i+1}^{\max}(t_2)}{\sigma_i^{\min}(t_2)}.
		\end{equation*}
		Since the estimate is independent of $t_1\geq T$ and singular values converge to LEs, the claim is proved.
	\end{proof}

	\subsection{Backward Phase}\label{subsectionBackwardPhase}
	Initial tuples for the backward phase are obtained from spaces spanned by vectors of the forward phase. Thus, it appears more practical to describe admissibility in terms of propagated forward vectors instead of $\left(\hat{v}(t_2)\right)$.
	\begin{lemma}\label[lemma]{lemmaAdmissibleAfterForwardPhase}
		Let $0<\delta(t)<\frac{1}{\sqrt{2}}$ with $\lambda\left(\frac{1}{\delta}\right)=0$ be given. There is $T>0$ such that for all $t_1,t_2\geq T$ and all $(b)\in\mathcal{A}d^{(\hat{v}^-(t_1))}(\delta(t_1))$ we have
		\begin{equation*}
		\mathcal{A}d_-^{(f)^r}\left(\sqrt{2}\delta(t_2)\right)\subset\mathcal{A}d_-^{\left(\hat{v}(t_2)\right)}(\delta(t_2)),
		\end{equation*}
		where $(f):=\mathcal{GS}\left(\Phi(t_1+t_2,\theta_{-t_1}\omega_0)b\right)$ and admissibility holds with respect to reversed degeneracies.
	\end{lemma}
	\begin{proof}
		Let $(f):=\mathcal{GS}\left(\Phi(t_1+t_2,\theta_{-t_1}\omega_0)b\right)$ for $(b)\in\mathcal{A}d^{(\hat{v}^-(t_1))}(\delta(t_1))$ be given, and let $(g)\in\mathcal{A}d_-^{\left(f\right)^r}\left(\sqrt{2}\delta(t_2)\right)$ be an admissible tuple. We estimate
		\begin{align*}
		d\left(\overline{U}_{i^-}^{(g)},\overline{U}_{i^-}^{\left(\hat{v}(t_2)\right)}\right)&\leq d\left(\overline{U}_{i^-}^{(g)},\overline{U}_{i^-}^{\left(f\right)^r}\right)+d\left(\overline{U}_{i^-}^{(f)^r},\overline{U}_{i^-}^{\left(\hat{v}(t_2)\right)}\right)\\
		&\leq\sqrt{1-2\delta(t_2)^2}+d\left(\left(\overline{U}_{p-i}^{(f)}\right)^{\perp},\left(\overline{U}_{p-i}^{\left(u(t_2)\right)}\right)^{\perp}\right)\\
		&=\sqrt{1-2\delta(t_2)^2}+d\left(\overline{U}_{p-i}^{\left(\Phi(t_1+t_2,\theta_{-t_1}\omega_0)b\right)},\overline{U}_{p-i}^{\left(u(t_2)\right)}\right).
		\end{align*}
		The last summand is bounded by 
		\begin{equation*}
			d(t_2):=\sup_{t_1\geq T}\, \sup_{(b)\in\mathcal{A}d^{(\hat{v}^-(t_1))}(\delta(t_1))}d\left(\overline{U}_{p-i}^{\left(\Phi(t_1+t_2,\theta_{-t_1}\omega_0)b\right)},\overline{U}_{p-i}^{(u(t_2))}\right)
		\end{equation*}
		for $t_2\geq T$ with $T$ as in \cref{lemmaProofStep1.2}. In particular, it holds $\lambda(d(t_2))<0$. Now, for $(g)$ to be $\delta(t_2)$-admissible w.r.t. $\left(\hat{v}(t_2)\right)$, it suffices to show that
		\begin{equation*}
		\sqrt{1-2\delta(t_2)^2}+d(t_2)\leq\sqrt{1-\delta(t_2)^2}
		\end{equation*}
		for $t_2$ large enough, which in turn is equivalent to 
		\begin{equation*}
		1-2\delta(t_2)^2+2\sqrt{1-2\delta(t_2)^2}d(t_2)+d(t_2)^2\leq 1-\delta(t_2)^2
		\end{equation*}
		and to
		\begin{equation*}
		\frac{d(t_2)\left(2\sqrt{1-2\delta(t_2)^2}+d(t_2)\right)}{\delta(t_2)^2}\leq 1.
		\end{equation*}
		The latter is true for $t_2$ large enough, since we have
		\begin{align*}
		&\lambda\left(\frac{d(t_2)\left(2\sqrt{1-2\delta(t_2)^2}+d(t_2)\right)}{\delta(t_2)^2}\right)\leq\lambda\left(\frac{d(t_2)\left(2+d(t_2)\right)}{\delta(t_2)^2}\right)<0.
		\end{align*}
	\end{proof}
	Next, we combine our characterization of the forward phase with backward propagation. During the backward phase, it is enough to restrict ourselves to tuples that have admissible extensions. A few arguments from the forward phase can be repeated by reversing the cocycle.
	\begin{lemma}\label[lemma]{lemmaProofStep2}
		Let $0<\delta(t)<\frac{1}{\sqrt{2}}$ with $\lambda\left(\frac{1}{\delta}\right)=0$ be given. It holds
		\begin{align*}
		&\overline{\lambda}\left(\sup_{(b)\in\mathcal{A}d^{(\hat{v}^-(t_1))}(\delta(t_1))}\,\sup_{(b')\in\left(\mathcal{A}d^{(f)^r}_{\textnormal{ext}_-}\left(\sqrt{2}\delta(t_2)\right)\right)^r}d\left(U_i^{\left(\Phi(-t_2,\theta_{t_2}\omega_0)b'\right)},E_i(\omega_0)\right)\right)\\
		&\hspace{1em}\leq -\min\left(|\lambda_{i}-\lambda_{i-1}|,|\lambda_{i}-\lambda_{i+1}|\right),
		\end{align*}
		where $(f):=\mathcal{GS}\left(\Phi(t_1+t_2,\theta_{-t_1}\omega_0)b\right)$.
	\end{lemma}
	\begin{proof}
		Applying \cref{lemmaProofStep1.1} to $\Phi$ and $\Phi^-$, we get
		\begin{equation*}
		\lambda\left(\sup_{(b)\in\mathcal{A}d^{(\hat{v}^-(t))}(\delta(t))}d\left(\overline{U}_i^{\left(\Phi(t,\theta_{-t}\omega_0)b\right)}, V_{p+1-i}^-(\omega_0)\right)\right)\leq -|\lambda_{i}-\lambda_{i+1}|
		\end{equation*}
		and
		\begin{equation*}
		\lambda\left(\sup_{(g)\in\mathcal{A}d_-^{(\hat{v}(t))}(\delta(t))}d\left(\overline{U}_{i^-}^{\left(\Phi(-t,\theta_{t}\omega_0)g\right)}, V_{p+1-i}(\omega_0)\right)\right)\leq -|\lambda^-_{i}-\lambda^-_{i+1}|.
		\end{equation*}
		By switching indices we can rewrite the latter as
		\begin{equation*}
		\lambda\left(\sup_{(g)\in\mathcal{A}d_-^{(\hat{v}(t))}(\delta(t))}d\left(\overline{U}_{(p+1-i)^-}^{\left(\Phi(-t,\theta_{t}\omega_0)g\right)}, V_{i}(\omega_0)\right)\right)\leq -|\lambda_{i}-\lambda_{i-1}|.
		\end{equation*}
		In short, we have exponentially fast converging approximations to $V_{p+1-i}^-(\omega_0)$ and $V_i(\omega_0)$, which are transversal subspaces with intersection $E_i(\omega_0)$ (see equation \cref{equationOseledetsSpacesIntersection}). Thus, we can apply \cref{corollaryConvergenceRateIntersections} to
		\begin{equation*}
		\mathcal{M}_t:=\left\{\left.\overline{U}_i^{\left(\Phi(t,\theta_{-t}\omega_0)b\right)}\ \right|\ (b)\in\mathcal{A}d^{(\hat{v}^-(t))}(\delta(t))\right\}
		\end{equation*}
		and
		\begin{equation*}
		\mathcal{N}_t:=\left\{\left.\overline{U}_{(p+1-i)^-}^{\left(\Phi(-t,\theta_{t}\omega_0)g\right)}\ \right|\ (g)\in\mathcal{A}d_-^{(\hat{v}(t))}(\delta(t))\right\}
		\end{equation*}
		to get a convergence rate estimate for intersections\footnote{Following this statement, one can prove convergence of algorithms that initiate randomly chosen vectors in the past and future, propagate them to the present, and then take intersections of involved subspaces to get an approximation of $E_i(\omega_0)$. Similar convergence theorems for continuous and discrete time can be derived.}:
		\begin{align*}
		&\overline{\lambda}\Bigg(\sup_{(b)\in\mathcal{A}d^{(\hat{v}^-(t_1))}(\delta(t_1))}\,\sup_{(g)\in\mathcal{A}d_-^{(\hat{v}(t_2))}\left(\delta(t_2)\right)}\\
		&\hspace{3em}d\left(\overline{U}_i^{\left(\Phi(t_1,\theta_{-t_1}\omega_0)b\right)}\cap\overline{U}_{(p+1-i)^-}^{\left(\Phi(-t_2,\theta_{t_2}\omega_0)g\right)},E_i(\omega_0)\right)\Bigg)\\
		&\hspace{1em}\leq -\min\left(|\lambda_{i}-\lambda_{i-1}|,|\lambda_{i}-\lambda_{i+1}|\right).
		\end{align*}
		By \cref{lemmaAdmissibleAfterForwardPhase} we can take the supremum over
		\begin{equation*}
		(g)\in\mathcal{A}d_-^{\left(f\right)^r}\left(\sqrt{2}\delta(t_2)\right)
		\end{equation*}
		instead, while maintaining the estimate. In particular, this is true for all admissible extensions $(g)$ of
		\begin{equation*}
		(b')^r\in\mathcal{A}d_{\textnormal{ext}_-}^{(f)^r}\left(\sqrt{2}\delta(t_2)\right).
		\end{equation*}
		Now, to prove the lemma it suffices to show that each admissible extension $(g)$ of 
		\begin{equation*}
		\left((b')^r_{(p+1-i)^-_1},\dots,(b')^r_{(p+1-i)^-_{d^-_{p+1-i}}}\right)=\left(b'_{i_{d_i}},\dots,b'_{i_{1}}\right)
		\end{equation*}
		satisfies
		\begin{equation*}
		U_i^{\left(\Phi(-t_2,\theta_{t_2}\omega_0)b'\right)}=\overline{U}_i^{\left(\Phi(t_1,\theta_{-t_1}\omega_0)b\right)}\cap\overline{U}_{(p+1-i)^-}^{\left(\Phi(-t_2,\theta_{t_2}\omega_0)g\right)}.
		\end{equation*}
		We clearly have
		\begin{equation*}
		U_i^{(b')}=U_{(p+1-i)^-}^{(b')^r}=U_{(p+1-i)^-}^{(g)}\subset\overline{U}_{(p+1-i)^-}^{(g)}
		\end{equation*}
		and hence
		\begin{equation*}
		U_i^{\left(\Phi(-t_2,\theta_{t_2}\omega_0)b'\right)}\subset\overline{U}_{(p+1-i)^-}^{\left(\Phi(-t_2,\theta_{t_2}\omega_0)g\right)}
		\end{equation*}
		for an admissible extension $(g)$. Moreover, the definition of extendable admissibility requires that
		\begin{align*}
		(b')^r_{(p+1-i)^-_1},\dots,(b')^r_{(p+1-i)^-_{d^-_{p+1-i}}}&\in U_{(p+1-i)^-}^{(f)^r}\oplus\dots\oplus U_{p^-}^{(f)^r}\\
		&= U_{i}^{(f)}\oplus\dots\oplus U_{1}^{(f)}\\
		&= \overline{U}_{i}^{\left(\Phi(t_1+t_2,\theta_{-t_1}\omega_0)b\right)}\\
		&=\Phi(t_2,\omega_0)\overline{U}_{i}^{\left(\Phi(t_1,\theta_{-t_1}\omega_0)b\right)},
		\end{align*}
		or equivalently, it holds
		\begin{equation*}
		U_i^{\left(\Phi(-t_2,\theta_{t_2}\omega_0)b'\right)}\subset\overline{U}_{i}^{\left(\Phi(t_1,\theta_{-t_1}\omega_0)b\right)}.
		\end{equation*}
		Thus, we have
		\begin{equation*}
		U_i^{\left(\Phi(-t_2,\theta_{t_2}\omega_0)b'\right)}\subset\overline{U}_i^{\left(\Phi(t_1,\theta_{-t_1}\omega_0)b\right)}\cap\overline{U}_{(p+1-i)^-}^{\left(\Phi(-t_2,\theta_{t_2}\omega_0)g\right)}.
		\end{equation*}
		Since admissible tuples are linearly independent, the left-hand side has dimension $d_i$. The right-hand side must have the same dimension for $t_1,t_2$ large enough, because the intersection converges to $E_i(\omega_0)$. Hence, we have equality of subspaces, which concludes the proof.
	\end{proof}

	\subsection{Proof of Theorems}\label{subsectionProofTheorems}
	\Cref{lemmaProofStep2} describes how admissible tuples fare in Ginelli's algorithm. The remaining work lies in connecting the lemma to measurement results from \cref{subsectionAdmissibility}.
	\begin{proof}[Proof of \textnormal{\cref{theoremGinelliConvergenceInMeasure}}]
		Fix $\epsilon>0$. By compactness of $\mathcal{K}$ we find $M>0$ with $\mathcal{K}\subset B_d(0,M)^d\times (B(M))^r$. Note that it is enough to prove the claim for the product of balls instead of $\mathcal{K}$. Furthermore, set $\delta(t):=\min\left(\frac{1}{t},\frac{1}{2\sqrt{2}}\right)$, so that $\lambda\left(\frac{1}{\delta}\right)=0$. Now, we use $\delta$ in \cref{lemmaProofStep2} to get
		\begin{equation*}
		\frac{1}{\min(t_1,t_2)}\log d\left(U_i^{\left(\Phi(-t_2,\theta_{t_2}\omega_0)b'\right)},E_i(\omega_0)\right)\leq -\min\left(|\lambda_{i}-\lambda_{i-1}|,|\lambda_{i}-\lambda_{i+1}|\right)+\epsilon
		\end{equation*} 
		for all $(b)\in\mathcal{A}d^{(\hat{v}^-(t_1))}(\delta(t_1))$ and $(b')\in\left(\mathcal{A}d^{(f)^r}_{\textnormal{ext}_-}\left(\sqrt{2}\delta(t_2)\right)\right)^r$ if $t_1$ and $t_2$ are large enough. Using the identification via $A^{(f)}$, we could equivalently assume $(b')=\left(A^{(f)}\alpha\right)$ for
		\begin{equation*}
		(\alpha)\in\left(\left(A^{(f)}\right)^{-1}\right)^d\left(\mathcal{A}d^{(f)^r}_{\textnormal{ext}_-}\left(\sqrt{2}\delta(t_2)\right)\right)^r=\left(\mathcal{A}d^{(e)^r}_{\textnormal{ext}_-}\left(\sqrt{2}\delta(t_2)\right)\right)^r.
		\end{equation*}
		Hence, it is enough to show that nonadmissible tuples have measure zero in the limit:
		\begin{align*}
		&\mu\left(\left(B_d(0,M)^d\times(B(M))^r\right)\setminus\left(\mathcal{A}d^{(\hat{v}^-(t_1))}(\delta(t_1))\times\left(\mathcal{A}d^{(e)^r}_{\textnormal{ext}_-}\left(\sqrt{2}\delta(t_2)\right)\right)^r\right)\right)\\
		&\hspace{1em}\leq\mu\left(B_d(0,M)^d\setminus\mathcal{A}d^{(\hat{v}^-(t_1))}(\delta(t_1))\right)\mu\left((B(M))^r\right)\\
		&\hspace{2em}+\mu\left(B_d(0,M)^d\right)\mu\left((B(M))^r\setminus\left(\mathcal{A}d^{(e)^r}_{\textnormal{ext}_-}\left(\sqrt{2}\delta(t_2)\right)\right)^r\right)\\
		&\hspace{1em}=\mu\left(B_d(0,M)^d\setminus\mathcal{A}d^{(e)}(\delta(t_1))\right)\mu\left(B(M)\right)\\
		&\hspace{2em}+\mu\left(B_d(0,M)^d\right)\mu\left(B(M)\setminus\mathcal{A}d^{(e)^r}_{\textnormal{ext}_-}\left(\sqrt{2}\delta(t_2)\right)\right).
		\end{align*}
		Here, we used invariance under orthogonal transformations of $B_d(0,M)$ to switch from $(\hat{v}^-(t_1))$ to $(e)$. By \cref{corollaryAdmissibleFiniteMeasure} and \cref{corollaryAdmissibleExtFiniteMeasure} the final estimate converges to zero as $\min(t_1,t_2)$ is increased. Hence, we get the desired convergence result.
	\end{proof}
	The discrete time version can be proved in a similar fashion.
	\begin{proof}[Proof of \textnormal{\cref{theoremGinelliConvergenceAlmostEverywhere}}]
		Assume discrete time $\mathbb{T}=\integer$ and $d>1$. We define $\delta_{\epsilon}(n):=\left(\frac{\epsilon}{\sqrt{2}n^2}\right)^{d-1}$ as our admissibility parameter satisfying $\lambda\left(\frac{1}{\delta_{\epsilon}}\right)=0$ for each $0<\epsilon<1$. Using $\delta_{\epsilon}$, we invoke \cref{lemmaProofStep2} to find that 
		\begin{equation*}
		\overline{\lambda}\left(d\left(U_i^{\left(\Phi(-n_2,\theta_{n_2}\omega_0)b'\right)},E_i(\omega_0)\right)\right)\hspace{1em}\leq -\min\left(|\lambda_{i}-\lambda_{i-1}|,|\lambda_{i}-\lambda_{i+1}|\right)
		\end{equation*}
		for $(b')=\left(A^{\mathcal{GS}\left(\Phi(n_1+n_2,\theta_{-n_1}\omega)b\right)}\alpha\right)$, whenever
		\begin{equation*}
		((b),(\alpha))\in\bigcap_{n_1,n_2\in\natural}\mathcal{A}d^{(\hat{v}^-(n_1))}(\delta_{\epsilon}(n_1))\times\left(\mathcal{A}d^{(e)^r}_{\textnormal{ext}_-}\left(\sqrt{2}\delta_{\epsilon}(n_2)\right)\right)^r.
		\end{equation*}
		This is true independent of our choice for $\epsilon$. Hence, it suffices to show that the complement of 
		\begin{equation}\label{equationSetProofConvergenceAE}
		\bigcup_{0<\epsilon<1}\,\bigcap_{n_1,n_2\in\natural}\mathcal{A}d^{(\hat{v}^-(n_1))}(\delta_{\epsilon}(n_1))\times\left(\mathcal{A}d^{(e)^r}_{\textnormal{ext}_-}\left(\sqrt{2}\delta_{\epsilon}(n_2)\right)\right)^r
		\end{equation}
		has measure zero\footnote{Note that the statement is not true in general for continuous time. In fact, in \cref{exampleRotatingOseledetsSpaces} no tuple $(b)$ is admissible w.r.t. $(\hat{v}^-(t_1))$ for all $t_1\in\real_{>0}$ simultaneously. Hence, in this case the set in \cref{equationSetProofConvergenceAE} would be empty.}, which can be proved by exhausting the domain of $((b),(\alpha))$ with products of balls: It holds
		\begin{align*}
		&\mu\Bigg(\left(B_d(0,M)^d\times(B(M))^r\right)\setminus\\
		&\hspace{3em}\bigcup_{0<\epsilon<1}\,\bigcap_{n_1,n_2\in\natural}\mathcal{A}d^{(\hat{v}^-(n_1))}(\delta_{\epsilon}(n_1))\times\left(\mathcal{A}d^{(e)^r}_{\textnormal{ext}_-}\left(\sqrt{2}\delta_{\epsilon}(n_2)\right)\right)^r\Bigg)\\
		&\hspace{1em}\leq\inf_{0<\epsilon<1}\Bigg(\sum_{n_1\in\natural}\mu\left(B_d(0,M)^d\setminus\mathcal{A}d^{(e)}(\delta_{\epsilon}(n_1))\right)\mu\left(B(M)\right)\\
		&\hspace{2em}+\sum_{n_2\in\natural}\mu\left(B_d(0,M)^d\right)\mu\left(B(M)\setminus\mathcal{A}d^{(e)^r}_{\textnormal{ext}_-}\left(\sqrt{2}\delta_{\epsilon}(n_2)\right)\right)\Bigg)\\
		&\hspace{1em}\leq\inf_{0<\epsilon<1}\left(\sum_{n_1\in\natural}\eta_1\left(\delta_{\epsilon}(n_1)\right)^{\frac{1}{d-1}}\mu(B(M))+\sum_{n_2\in\natural}\mu\left(B_d(0,M)^d\right)\eta_2\left(\sqrt{2}\delta_{\epsilon}(n_2)\right)^{\frac{1}{d-1}}\right)\\
		&\hspace{1em}=\inf_{0<\epsilon<1}\epsilon\,\left(\sum_{n\in\natural}\frac{\eta_1\mu(B(M))+\eta_2\sqrt{2}^{\frac{1}{d-1}}\mu\left(B_d(0,M)^d\right)}{\sqrt{2}n^2}\right)\\
		&\hspace{1em}=0
		\end{align*}
		for all $M>0$. Here, it was crucial to use \cref{propositionAdmissibleMeasureEstimate} and \cref{propositionAdmissibleExtMeasureEstimate} to get a more precise measure estimate on nonadmissible tuples.		
	\end{proof}
	\bigskip
	
		\section{Conclusions}\label{sectionConclusions}
	We defined Ginelli's algorithm as a means to compute CLVs/Oseledets spaces, which are the most natural choice for directions describing asymptotic expansion and contraction in the tangent linear model along a given trajectory. The existence of those characteristic directions was provided by the MET of Oseledets. Moreover, the theorem handed us an interface able to link CLVs with a limit of finite time scenarios, in which Ginelli's algorithm is applied to initial vectors. It turned out that certain configurations of initial vectors perform better than others given the same runtime, whereas in some cases the algorithm would not even converge - a problem that did not receive enough attention in previous attempts to prove convergence.\par
	As a measure to tackle this problem, we introduced the concept of admissibility. A configuration of initial vectors is called admissible if it is not too far from the optimal initial vectors, i.e. right singular vectors of the propagator. The term ``not too far'' was made more precise by a parameter $\delta\in[0,1]$. In our formulation, $\delta$-values close to one imply a good correlation, whereas small values of $\delta$ stand for greater distances to the configuration of singular vectors.\par
	In \cite{Ershov1998} it is shown that configurations with $\delta>0$ will align with left singular vectors when propagated from the present state to future states. While the condition of admissibility depends on the chosen runtime, according to the MET the configuration of right singular vectors defining admissibility at the present state converges. Using the limit configuration, it is possible to show that almost all initial configurations will yield a good approximation to left singular vectors if propagated long enough from the present state. However, in Ginelli's algorithm initial vectors are first propagated from past states to the present state. In this case, left singular vectors converge to an orthonormalization of CLVs. But, as the admissibility condition depends on right singular vectors at the past state, the set of admissible initial vectors varies with the runtime and in general does not converge to a limit set. In fact, we presented an example where no fixed initial configuration is admissible for all past states simultaneously. Consequently, the continuous time version of Ginelli's algorithm cannot be expected to converge for fixed initial configurations in general. Instead, we have shown convergence in measure of the continuous time version by carefully analyzing the time-sensitivity of propagated vectors. Moreover, due to suitable measure estimates for sets of admissible vectors, we were able to prove convergence for almost all initial vectors in the discrete time case.\par
	The convergence results for both time cases relate the speed of convergence to LEs. Using the Lyapunov index notation, we were able to prove that Ginelli's algorithm converges exponentially fast with a rate given by the minimum distance between LEs. Interestingly, this was already predicted and observed in applications.\par
	It is important to point out that the Lyapunov index notation neglects system-dependent prefactors for the speed of convergence on subexponential timescales, which may very well be of importance for limited time scenarios. Yet, if enough data is available, subexponential factors, e.g. from choosing two different initial conditions, can be ignored. Moreover, nonadmissible initial configurations will in general turn admissible due to numerical noise. Hence, the concept of admissibility and the different versions of convergence do not play a noticeable role in practice. They can be seen rather as tools or as products of a precise mathematical proof of convergence.\par
	While the proof assumes perfect computations, it is often not known how perturbed data affects LEs and CLVs. In particular, the possibly noncontinuous dependence of the Lyapunov spectrum on the choice of trajectory adds to the uncertainty. In this regard, it would be interesting to know more about how perturbations affect the outcome of Ginelli's algorithm in numerical simulations as well as in analytical computations.\par
	Ultimately, a wide range of applications, some of which are referenced here, underline the importance of CLVs for dynamical systems. Our convergence proof not only verifies the use of Ginelli's algorithm in those applications, but encourages one to apply the concept of CLVs to further scenarios. Moreover, the tools obtained during the proof can be used to investigate other algorithms, such as Wolfe-Samelson's algorithm \cite{Wolfe2007}, as well. In general, we expect our rigorous mathematical treatment to enable a more in-depth analysis that will lead to new insights and improvements of CLV-algorithms, which are important instruments to finding structure in the chaos of dynamical systems.	
	\bigskip
	
	\section*{Acknowledgments}
	This paper is a contribution to the project M1 (Instabilities across scales and statistical mechanics of multi-scale GFD systems) of the Collaborative Research Centre TRR 181 "Energy Transfer in Atmosphere and Ocean" funded by the Deutsche Forschungsgemeinschaft (DFG, German Research Foundation) - Projektnummer 274762653. Special thanks goes to my colleagues from project M1, in particular, R. Lauterbach and I. Gasser for frequent feedback and S. Schubert for motivational discussions on applications of CLVs.
	\bigskip
	

\bibliographystyle{siam}
\bibliography{sources}
	
\end{document}